\documentclass[a4paper,11pt,english,british,reqno]{article}
\usepackage{setspace,graphicx,
epstopdf,amsmath,amsfonts,amsgen, mathtools,
amstext,amsthm,amsbsy,amsopn,amssymb,
bbm,tikz,
parskip,verbatim,mathrsfs,enumerate,
xcolor,comment}  
\usepackage[utf8]{inputenc}   
\usepackage[top=2.5cm, bottom=2.5cm, left=1.6cm, right=1.6cm]{geometry}
\usepackage[square,numbers]{natbib} 
\usepackage{todonotes}
\usepackage[pdfborder={0 0 0}]{hyperref}

\def\be{\boldsymbol{e}}
\def\bh{\boldsymbol{h}}

\def\b0{\boldsymbol{0}}

\newcommand{\R}     {\mathbb{R}} 
\newcommand{\Z}     {\mathbb{Z}} 
\newcommand{\N}     {\mathbb{N}}

\newcommand{\E}     {\mathbb{E}} 
 
\newcommand{\T}     {\mathbb{T}}

\newcommand{\Acal}   {{\mathcal A }}
\newcommand{\Bcal}   {{\mathcal B }}
\newcommand{\Ccal}   {{\mathcal C }} 
 
\newcommand{\Ecal}   {{\mathcal E }} 
 
\newcommand{\Gcal}   {{\mathcal G }} 
\newcommand{\Hcal}   {{\mathcal H }}

\newcommand{\Mcal}   {{\mathcal M }}

\newcommand{\Pcal}   {{\mathcal P }}

\newcommand{\Scal}   {{\mathcal S }}

\newcommand{\Vcal}   {{\mathcal V }} 
\newcommand{\Wcal}   {{\mathcal W }}

\newcommand{\Exp}{\mathscr{E}\kern-0.2mm{\operatorname{xp}}}
\newcommand{\Log}{\mathscr{L}\kern-0.2mm{\operatorname{og}}}

\def\1{{\mathchoice {1\mskip-4mu\mathrm l}      
{1\mskip-4mu\mathrm l} 
{1\mskip-4.5mu\mathrm l} {1\mskip-5mu\mathrm l}}}

\numberwithin{equation}{section}
\numberwithin{figure}{section}
\newtheoremstyle{plain}
  {6pt}
  {4pt}
  {\slshape}
  {}
  {\bfseries}
  {.}
  {0.5em}
  {}%
\newtheorem{thm}{\protect\theoremname}
  \newtheorem{defn}[thm]{\protect\definitionname}
  
  \newtheorem{prop}[thm]{\protect\propositionname}
  \newtheorem{rem}[thm]{\protect\remarkname}
  
  \newtheorem{lem}[thm]{\protect\lemmaname}
  \numberwithin{thm}{section}

\usepackage{babel}
  \addto\captionsbritish{\renewcommand{\corollaryname}{Corollary}}
  \addto\captionsbritish{\renewcommand{\definitionname}{Definition}}
  \addto\captionsbritish{\renewcommand{\factname}{Fact}}
  \addto\captionsbritish{\renewcommand{\propositionname}{Proposition}}
  \addto\captionsbritish{\renewcommand{\remarkname}{Remark}}
  \addto\captionsbritish{\renewcommand{\theoremname}{Theorem}}
  \addto\captionsenglish{\renewcommand{\corollaryname}{Corollary}}
  \addto\captionsenglish{\renewcommand{\definitionname}{Definition}}
  \addto\captionsenglish{\renewcommand{\factname}{Fact}}
  \addto\captionsenglish{\renewcommand{\propositionname}{Proposition}}
  \addto\captionsenglish{\renewcommand{\remarkname}{Remark}}
  \addto\captionsenglish{\renewcommand{\theoremname}{Theorem}}
  \providecommand{\corollaryname}{Corollary}
  \providecommand{\definitionname}{Definition}
  \providecommand{\factname}{Fact}
  \providecommand{\propositionname}{Proposition}
  \providecommand{\remarkname}{Remark}
\providecommand{\theoremname}{Theorem}
\providecommand{\lemmaname}{Lemma}

\title{Site monotonicity and uniform positivity for interacting random walks 
and the spin $O(N)$ model with arbitrary $N$}

\author{Benjamin Lees\footnote{lees@mathematik.tu-darmstadt.de} \and Lorenzo Taggi\footnote{lorenzo.taggi@gmail.com}}
\date{}

\begin{document}

\maketitle

\begin{abstract}
We provide a  uniformly-positive point-wise lower bound for the two-point  function of the classical spin $O(N)$ model on the torus of $\mathbb{Z}^d$, $d \geq 3$, when $N \in \mathbb{N}_{>0}$  and the  inverse temperature $\beta$ is large enough.
This is a new result when $N>2$ and  extends the classical result of Fr\"ohlich, Simon and Spencer (1976).
Our bound follows from a new site-monotonicity property of the two-point function which is of independent interest and holds not only for the spin $O(N)$ model with arbitrary $N \in \mathbb{N}_{>0}$, but for a wide class of systems of interacting random walks and loops, including the loop $O(N)$ model,  random lattice permutations,  the dimer model, the double-dimer model, and the loop representation of the classical spin $O(N)$ model.
\end{abstract}

\section{Introduction}
We consider a system of interacting random loops and walks that reduces to several paradigmatic models in statistical mechanics for specific choices of the parameters, such as the \textit{loop $O(N)$ model},  \textit{random lattice-permutations}, the \textit{double-dimer model},
the \textit{dimer model},
and a representation of the \textit{classical spin $O(N)$ model}.
The spin $O(N)$ model is the most well known of these models, it involves the vertices of a graph carrying (classical) spins in $\mathbb{S}^{N-1}\subset\R^{N}$ that interact via their inner-product. The case $N=1$ is the \textit{Ising model}, the case $N=2$ is the XY or \textit{rotator model}, and the case $N=3$ is the \textit{classical Heisenberg model} (see \cite{FriedliVelenik} for an overview).
The loop $O(N)$ model is related to the spin $O(N)$ model and, in two dimensions, it is conjectured to converge to SLE in an appropriate sense under the correct scaling and choice of parameters. It exhibits a very rich phase diagram, (see 
\cite{PeledSpinka} for an overview).
The study of random lattice-permutations is motivated by its connections to the quantum Bose gas \cite{Feynman}. In particular, the occurrence of infinite cycles in such permutations is related to the occurrence of Bose-Einstein condensation \cite{Ueltschi}.
The dimer model goes back to the work of Kasteleyn \cite{Kasteleyn} and Temperley-Fisher \cite{Temperley} and is closely connected to the study of perfect matchings of a graph. It is the subject of an extensive physical and mathematical literature, we refer the reader to \cite{Kenyon2} for a relatively recent discussion.

\paragraph{Uniform positivity.}
Consider the spin $O(N)$ model.
In the famous work of Fr\"ohlich, Simon and Spencer \cite{FrohlichSimonSpencer} it was shown that, for the torus $\T_L=\Z^d/L\Z^d$ with $d\geq 3$ and inverse temperature $\beta$ large enough, spin correlations do not decay with the distance between the sites. This established the occurrence of a phase transition. 
More precisely, when $d \geq 3$, there exists a finite $\beta_0=\beta_0(d,N) < \infty$ such that,
\begin{equation}\label{eq:classicalstatement}
\liminf_{ L \rightarrow \infty} \quad \frac{1}{|\T_{2L}|}\sum\limits_{z \in \T_{2L}} \langle \varphi_o \cdot \varphi_z \rangle_{2L, N, \beta} \geq 1 - \frac{\beta_0}{\beta},
\end{equation}
(see Theorem \ref{theo:classical} for a precise formulation, this particular formulation first appears in \cite{DLS}).
The result does not imply that 
the two-point correlation,
$\langle \varphi_o \cdot \varphi_z \rangle_{2L, N, \beta}$,
is bounded away from zero uniformly in $L$ and in the choice of the site $z$. This has been proved for N=1 using Peierls' argument \cite{Peierls} and N=2 using the Messager, Miracle-Sole inequality \cite{MMS} and it is expected to be true for any $N\in\N_{>0}$.
Our first main result, Theorem \ref{theo:pointwise}, 
rigorously establishes this for arbitrary $N \in \mathbb{N}_{>0}$ when $\beta$ is large and $d \geq 3$. More precisely, take $\T_L=\Z^d/L\Z^d$ with $d \geq 3$ and nearest-neighbour edges and an arbitrary $N \in \mathbb{N}_{>0}$. We prove that there is a $\beta_0=\beta_0(d,N)<\infty$ such that if $\beta>\beta_0$ then there exists a  constant $C = C(\beta) \in (0, 1]$ such that,
\begin{equation}\label{eq:liminfpositive}
\forall z \in \mathbb{Z}^{d}
\quad 
\liminf\limits_{L \rightarrow \infty } \,  
 \langle  \varphi_o \cdot  \varphi_z \rangle_{2L, N, \beta} > C.
\end{equation}

\paragraph{Site-monotonicity.}
Our  main result is a consequence of a new  \textit{site monotonicity property}, Theorem \ref{theo:monotonicity}, which holds for a general soup of interacting random loops and walks that we call the \textit{random path model} (RPM) and which includes all the models mentioned above. It extends the site-monotonicity property of Messager and Miracle-Sole \cite{MMS}, which applied to the spin $O(N)$ model with $N=1,2$.

The RPM can be informally defined as follows.
A realisation of the RPM is an ensemble of an arbitrary collection of open and closed nearest-neighbour paths, which we will refer to as \textit{walks} and \textit{loops}, respectively
(see Figure \ref{Fig:soupexample} for an example of such an ensemble).
\begin{figure}
\includegraphics[scale=0.26]{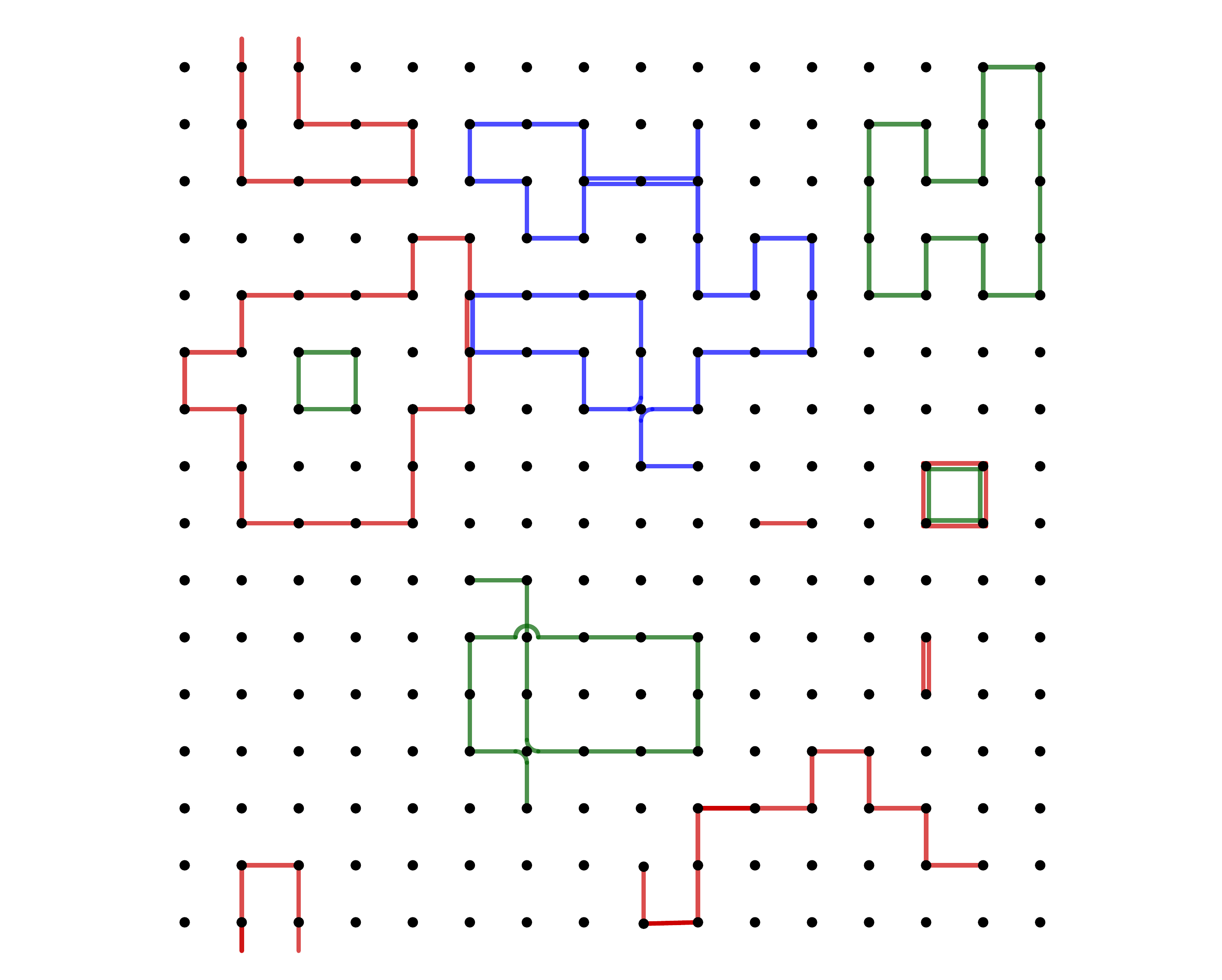}
\centering
\caption{A realisation of the random path model. Edges that are crossed multiple times are represented by narrowly separated parallel lines.
}
\label{Fig:soupexample}
\end{figure}
For $N \in \mathbb{N}_{>0}$ a `colour' in $\{1,\dots,N\}$ is assigned to each path. Realisations consisting of the same paths but different colour assignments are distinguished. The weight of a realisation, denoted by $w$, is proportional to 
\begin{equation}\label{eq:informaldefinition}
\beta^{ \footnotesize \mbox{ total length of paths in $w$} }  \, \, \prod_{x \in \T_L}    U_x(w) ,
\end{equation}
where $\beta\geq 0$ and $U_x(w)$ is a non-negative \emph{weight function} depending on how many times a walk or a loop of each colour visits the vertex $x \in \T_L$.
All of the models mentioned above can be obtained from specific choices of $U=(U_x)_{x \in \T_L}$, thus the setting we introduce allows the comparison of all such models in a unified framework.
The central quantity we consider is the \textit{two-point function}, 
$
G_{L, N, \beta, U}(x,y),
$
where $x, y \in \T_L$, $x\neq y$.
Informally, it is the ratio between the weight of realisations with one unique `long' walk connecting $x$ and $y$ and the weight of realisations without any such `long' walks (`short' walks consisting of a single edge
which we call \textit{dimers}, might be present in both terms, $(U_x)_{x \in \T_L}$ allowing). 
The decay or not of $G_ {L, N, \beta, U}(x,y)$ with $|y-x|$ in the limit $L\to\infty$ tells us whether or not the model exhibits long-range order. 
Additionally, for some choices of $U = (U_x)_{x \in \T_L}$, $G_{L, N, \beta, U}(x,y)$ corresponds to the spin-spin correlation of another model (see below). When such a correspondence is available, one might use methods from one model to answer questions about the other.
Theorem \ref{theo:monotonicity} states several monotonicity properties of the two-point function
of  all models mentioned above, for any value of the parameters.
It states that the two-point function between the point $o=(0,\dots,0)$ and an arbitrary `odd point' on the torus does not decrease if we project such a point onto an arbitrary coordinate axis and that the two-point function between  $o$ and  `odd' points lying on a cartesian axis is non-increasing with the distance of the point from $o$.
From this we deduce, for example, that the two-point function between two arbitrary sites, $x, y \in \T_L$, that differ by an odd amount in one of the coordinate directions, is bounded from above by the two-point function 
between two neighbouring sites.
In other words, the most convenient thing for the system is that such a long walk interacting with the ensemble of  loops and dimers ends at a neighbour of its starting point, resembling a loop or a dimer itself.

\paragraph{Methodology.}
The essential feature of the weights (\ref{eq:informaldefinition}) (which follows from our general assumptions in Definition \ref{def:measure}) is that they can be expressed as a product of `identical'  `local' functions.
Due to this important property and torus symmetries, the measure can be proved to be \textit{reflection positive for reflections `through edges'}.
Reflection positivity is the key tool that we employ in this paper.
It is a classical tool for the analysis of spin systems and it was also used in \cite{Chayes}
in the context of loop soups.
Using this property we obtain a new inequality involving two-point functions. Such an inequality can be viewed as a new application of reflection positivity and all our results are derived from it.
The inequality states the following.
Consider the torus $\T_L=\Z^d/L\Z^d$ in dimension $d\geq 2$ with $L\in2\N$ and nearest-neighbour edges. Let $\Theta$ be a reflection of sites of $\T_L$ in a plane $R$ bisecting edges and perpendicular to one of the cartesian vectors. $R$ identifies two disjoint sets, $\T_L^+$ and $\T_L^-$ with $\T_L=\T_L^+\cup\T_L^-$, such that $\Theta(\T_L^{\pm}) = \T_L^{\mp}$.
\begin{figure}
\includegraphics[scale=0.22]{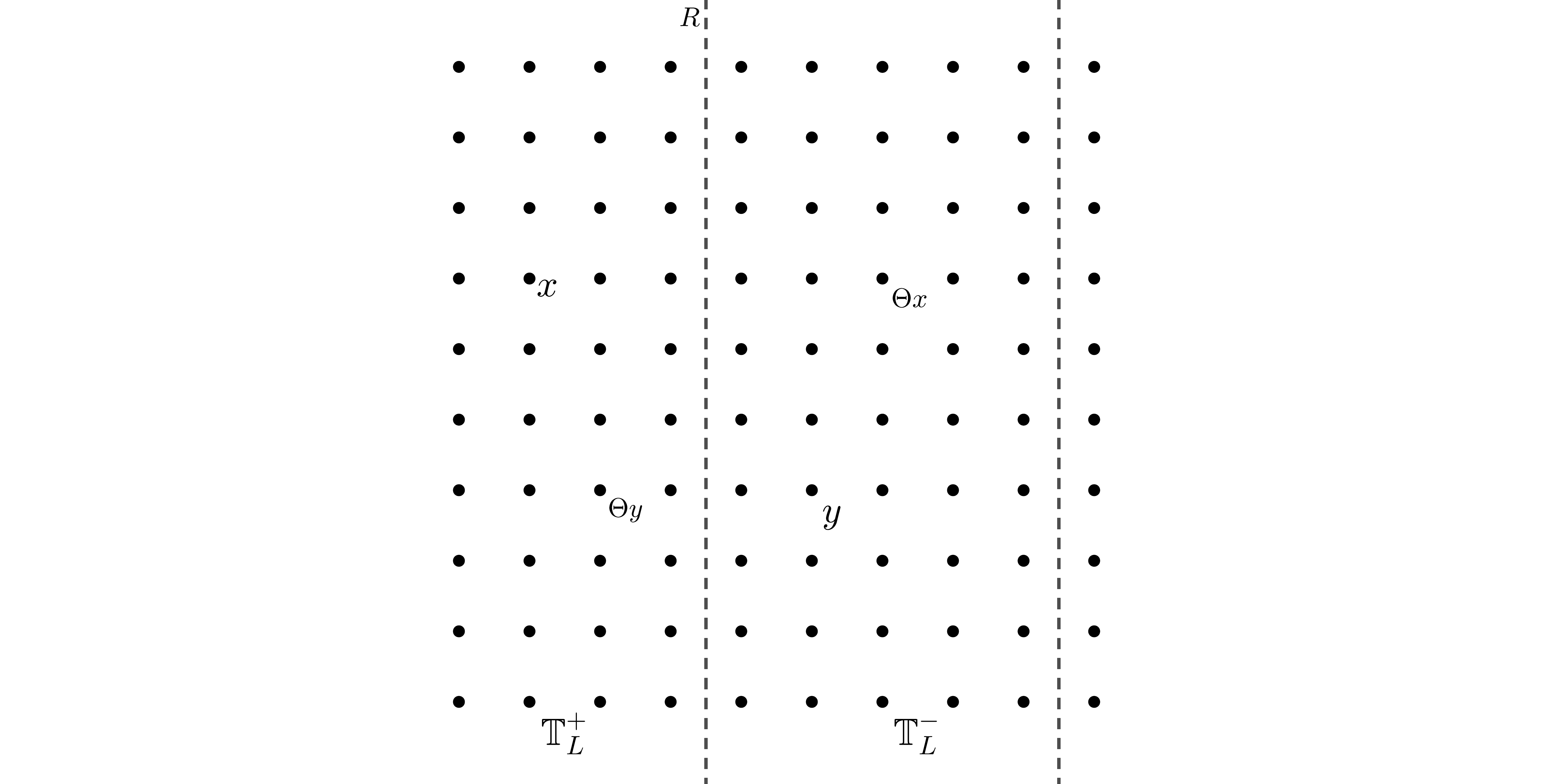}
\centering
\caption{Illustrations of two points $x,y \in \T_L$ and of their reflection.
}
\label{Fig:reflectionsexample}
\end{figure}
Our inequality states that, for any such reflection $\Theta$, $x\in\T_L^+$, $y\in\T_L^-$, and $U = (U_x)_{x \in \T_L}$ satisfying the condition in Definition \ref{def:measure}, we have that
\begin{equation}\label{eq:stat1corrineq}
G_{L, N, \beta, U}(x,y)  \leq \tfrac12  \Big ( \, G_{L, N, \beta, U}\big (x, \Theta(x)\big ) +G_{L, N, \beta, U}\big (\Theta(y), y \big ) \, \Big ).
\end{equation}
Our general  monotonicity properties follow from an iterative application of this inequality, or a generalisation that involves an average over two-point functions,
and an argument by contradiction.

Our site monotonicity properties and the bound (\ref{eq:classicalstatement}) lead to the uniformly positive point-wise lower bound for the two-point function of the spin $O(N)$ model. If such a bound was derived for other models that we mentioned above, then our monotonicity result would also imply a corresponding 
result for these models.

\paragraph{Question.}
Derive an infrared bound for the loop $O(N)$ model, random lattice permutations, or double dimer model
when $d \geq 3$ and $\beta$ is large enough.

We shall end this section by describing the organisation of this paper.
In Section \ref{sect:definition} we present the rigorous definition of the random path model, we show that this model reduces to, or is a representation of, other models when the weight function $U$ is chosen appropriately, and we state our results formally. In Section \ref{sec:RP} we introduce the main technique, reflection positivity.
In Section \ref{sec:MonAndProofs} we use this technique  to derive our main theorems, Theorem \ref{theo:monotonicity}
and Theorem \ref{theo:pointwise}.

\section*{Notation}
\begin{center}
	\begin{tabular}{ l l }
 $N \in \mathbb{N}_{>0}$& the number of colours\\

$\mathcal{G} = ( \mathcal{V}, \mathcal{E})$ &  an undirected, simple, finite graph \\

$e \in \mathcal{E}$ or $\{x,y\} \in \mathcal{E}$ & undirected edges \\


$x \sim y$ & two neighbour vertices, i.e, $x, y \in \mathcal{V}$ such that $\{x,y\} \in \mathcal{E}$ \\

$\partial^e D$ & $ \{v \in \mathcal{V} \, \, :  v \not\in D \mbox{ and }  z \sim v \mbox{ for some } z \in D\}$ \\

$\partial^i D$ & $ \{v \in \mathcal{V} \, \, :  v \in D \mbox{ and } \exists z \sim v \mbox{ s.t. } z \not\in D\}$\\
 
$\mathcal{M}_{\mathcal{G}}$ & set of link cardinalities on $\mathcal{G}$ \\

$\mathcal{C}_{\mathcal{G}}(m)$ & the set of colourings for  $m\in\mathcal{M}_{\mathcal{G}}$ \\

$\mathcal{P}_{\mathcal{G}}(m, c)$ & the set of pairing configurations for $m\in\mathcal{M}_{\mathcal{G}}$ and $c\in\mathcal{C}_{\mathcal{G}}(m)$ \\
 
$w= (m,c, \pi)$ &  wire configuration with $m \in\mathcal{M}_{\mathcal{G}}$, $c \in\mathcal{C}_{\mathcal{G}}(m)$, and
$\pi \in \mathcal{P}_{\mathcal{G}}(m, c)$\\ 

$\mathcal{W}_G$ &  the set of wire configurations on $\mathcal{G}$ \\

$n^i_x(w)$ & local occupancy of $i$-links \\

$n_x(w)$ & $ \sum_{i=1}^{N} n_x^i$\\

$u^i_x(w)$ & number of walks of colour $i$ having $x$ as an extremal vertex\\


$v^i_x(w)$ & number of pairings of i-links at $x$\\

$t_x(w)$ & number of links incident to  $x$ which are paired at $x$ to another link touching  $x$\\


$\beta \in \mathbb{R}_{\geq 0}$ & {inverse temperature} \\

$U = (U_x)_{x \in \mathcal{V}}$ &  {weight function} \\

$Z_{\Gcal,N,\beta,U}(x,y) \,  \mbox{or } \, Z_{L}(x,y)$ & {weight of configurations with a 1-path from $x$ to $y$}\\

$Z_{L}(z)$ & $Z_{L}(o,z)${, where }$o${ is the origin of the torus}\\

$G_{\Gcal,N,\beta,U}(x,y) \,  \mbox{or } \, G_{L}(x,y)$ & {the two-point function between $x$ and $y$ in the random path model} \\

$G_{L}(z)$ & $G_{L}(o,z)${, where }$o${ is the origin of the torus}\\

$\langle  \varphi_o \cdot  \varphi_z \rangle_{L, N, \beta}$ & {the two-point correlation between $x$ and $y$ in the spin $O(N)$ model}\\

 
\end{tabular}

\end{center}

\section{Definitions and results}
\label{sect:definition}
In this section we define the random path model (RPM), we show that, under specific choices of the weight function, this model reduces to other paradigmatic well studied models, for example random lattice permutations, the loop $O(N)$ model,  the double-dimer model, the dimer model, and a representation of the spin $O(N)$ model, and we
state our main results formally.
The  model we introduce is closely related to the one which was introduced in  \cite{BenassiUeltschi}, from which we borrow part of the notation.
Note however that our exposition presents some important differences with respect to  \cite{BenassiUeltschi}. The first difference is that in our framework an arbitrary number of walks are present, which is an essential aspect for obtaining our main theorems.  The second difference is that a colour is assigned to each path and  two realisations consisting of the same paths but  different colour assignments are distinguished. This allows us to introduce a weight function which depends on the colour of the path. Such a generalisation means our model generalises the model of lattice permutations, the double-dimer model and the dimer model, which can be seen by choosing the parameters and the weight function appropriately.

\subsection{Definitions}\label{sect:modeldefinition}
Let $\mathcal{G} = ( \mathcal{V}, \mathcal{E})$ be an undirected, simple, finite graph, and let $N \in \mathbb{N}_{>0}$. We will refer to $N$ as the \textit{number of colours}. 
A  realisation of the random path model can be viewed as a collection of undirected paths (which might be closed or open).
To define a realisation we need to introduce links and pairings. A path is identified by a collection of links, a colouring function and by pairings.
A configuration of links is denoted by 
 $m \in  \mathcal{M}_{\mathcal{G}} := \mathbb{N}^{\mathcal{E}}$. More specifically
$$m = \big ( m_e \big )_{e \in \mathcal{E}},$$
where $m_{e}\in\N$ represents the number of links  on the edge $e$.
No constraint concerning the parity of $m_e$ is introduced.

Given a link configuration $m \in \mathcal{M}_{\mathcal{G}}$, a \textit{colouring} $c \in \mathcal{C}_{\mathcal{G}}(m) := \{1, \ldots, N\}^{m}$ is a realisation which assigns an integer in $\{1, \ldots, N\}$ to each link, which will be called its \textit{colour}.
More precisely, 
$$
c = (c_e)_{e \in \mathcal{E}},
$$
is such that $c_e \in \{1, \ldots, N\}^{m_e}$, where $c_e(p) \in \{1, \ldots, N\}$ is the colour of the $p$-th link which is parallel to the edge $e \in \mathcal{E}$.
See Figure \ref{Fig:pairingexample} for an example.
We will call \textit{i-link} any link which gets colour $i \in \{1, \ldots, N\}$.

Given a link configuration, $m \in \mathcal{M}_{\mathcal{G}}$, and a colouring $c \in \mathcal{C}_{\mathcal{G}}(m)$, a pairing $ \pi = (\pi_x)_{x \in \mathcal{V}}$ for $m$ and $c$ 
pairs links touching $x$ (i.e. links on edges incident to $x$) in such a way that,  if two links are  paired, then they have the same colour. A link touching $x$ can be paired to at most one other link touching $x$, and
it is not necessarily the case that all links touching $x$ are paired to another link at $x$.
Given two links, if there exists a vertex $x$ such that such links are \textit{paired at $x$}, then we say that such links are \textit{paired}. It follows from these definitions that a link can be paired to at most two other links. We remark that by definition a link cannot be paired to itself. We denote by $\mathcal{P}_{\mathcal{G}}(m, c)$ the set of all such pairings for $m\in\mathcal{M}_\Gcal$, $c \in \mathcal{C}_{\mathcal{G}}(m)$. 
A wire configuration is an element $w = ( m, c, \pi)$
such that $m \in \mathcal{M}_{\mathcal{G}}$,
$c \in \mathcal{C}_{\mathcal{G}}(m)$,  
$\pi \in \mathcal{P}_{\mathcal{G}}(m, c)$.
We let $\mathcal{W}_{\mathcal{G}}$ be the set of  wire configurations. 
It follows from these definitions that any $w \in \mathcal{W}_{\mathcal{G}}$ can be viewed as a collection of loops and of walks, as in Figure \ref{Fig:soupexample}.

Loops and walks have no starting point and no orientation and they are formally defined as equivalence classes of directed walks.
Given $w = (m, c, \pi) \in \mathcal{W}_{\mathcal{G}}$ and $e \in \mathcal{E}$, let $(e, p)$ be the 
p-th link at $e$, where $p \in \{1, \ldots,  m_e\}$.
A \textit{directed walk} of colour $i$ is  an ordered set of links, $( (e_1, p_1), (e_2, p_2),$ $ \ldots, (e_\ell, p_\ell))$, 
where $e_j \in \mathcal{E}$ and $p_j \in \{1, \ldots, m_{e_j}\}$,
such that $(e_j, p_j)$ is paired to $(e_{j+1}, p_{j+1})$ for each $j\in\{1, \ldots, \ell-1\}$
and  each link has colour $i$.
Such a sequence is said to be \textit{closed} if  $\ell > 2$ and $(e_\ell, p_\ell)$ is paired to  $(e_1, p_1)$ or if $\ell = 2$ and $(e_\ell, p_\ell)$ and  $(e_1, p_1)$ are paired to each other at both their end-points.
If such a sequence is not   \textit{closed}, then it is considered \textit{open.}
Two directed closed walks are said to be \textit{equivalent} if they are the same colour and it is possible to map one sequence into the other through an inversion and/or a cyclic permutation of the sequence.
A \textit{loop} of colour $i$ is an equivalence class of directed closed walks of colour $i$.
Two directed open walks are said to be equivalent if they are the same colour and if one can map one sequence into the other through an inversion of the sequence.
A \textit{walk} of colour $i$ is an equivalence class of directed open walks of colour $i$.
When referring to loops and walks, we will not always specify their colour. More generally, loops and walks will  be called  \textit{paths}.

\begin{figure}
\includegraphics[scale=0.26]{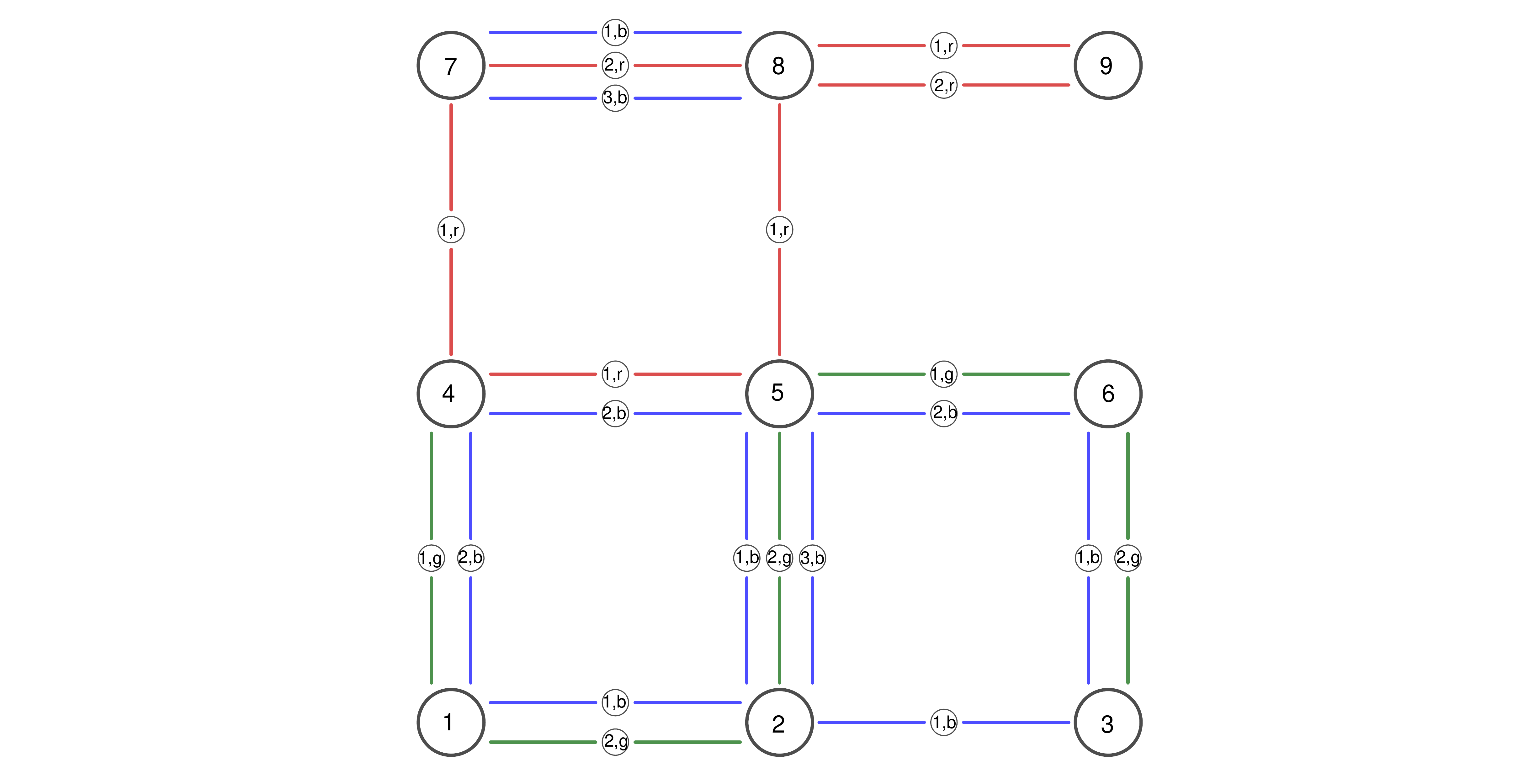}
\centering
\caption{
A pair  $(m, c)$, $m \in \mathcal{M}_{\mathcal{G}}$,
$c \in \mathcal{C}_{\mathcal{G}}(m)$,  on the graph $\{1, 2, 3\} \times \{1,2,3\}$. For example, two links connect the vertex $4$ to the vertex $5$, the \textit{first} link is coloured \textit{red} and the \textit{second} link is coloured blue (here, colours are identified by a letter, $r$, $b$ or $g$).
Pairings, which establish which links are paired at each vertex, are not represented in the figure.
 There are several wiring configurations which are compatible with the pair $(m, c)$ in the figure.
 Since the link cardinalities at some vertices are odd, at least one walk must be present in each such wiring configuration. This figure is analogous to \cite[Figure 1]{BenassiUeltschi}.
}
\label{Fig:pairingexample}
\end{figure}

We let $u^i_x(w)$ be the the number of $i$-links touching $x$ which are not paired to any other link at $x$. In other words, this number corresponds to the number of times a walk starts or ends at $x$.
Let $v^i_x(w)$ be the number of $i$-links touching $x$ which are paired at $x$ to another $i$-link touching $x$, then $v^i_x(w)/2$ corresponds to the number of pairings of $i$-links at $x$.
Set 
\begin{align}\label{eq:numerofihits}
n^i_x(w) & := \frac{v^i_x(w)}{2} \,   +  \,  u^i_x(w).
\end{align}
In other words, $n^i_x$ corresponds to the number of times $x$ is visited by a loop or by a walk of colour $i$.
We define
\begin{equation}\label{eq:numerofhits}
n_x(w)=\sum_{i=1}^Nn^i_x(w),
\end{equation}
to be the total number of times $x$ is visited by a loop or walk. We also write
\begin{equation}
\mathbf{n}_x(w)=(n^1_x(w),\dots,n^N_x(w)), \qquad \mathbf{u}_x(w)=(u^1_x(w),\dots,u^N_x(w)).
\end{equation}
Additionally we define $t_x(w)$ to be the number of links incident to $x$ that are paired at $x$ with another link on the same edge. To obtain, for example, the $O(N)$ loop model we need to restrict to configurations where $t_x(w)=0$ for each $x$. See section \ref{sect:special cases} for details.

\begin{defn}\label{def:measure}
Given $N \in \mathbb{N}_{>0}$, an inverse temperature  $\beta \in \mathbb{R}_{\geq 0}$, and
a \textit{weight function} $ U : \mathbb{N}^{2N+1} \rightarrow \mathbb{R}_{\geq 0}$, 
we define the non-negative measure $\mu_{\Gcal, N, \beta, U}$ by
\begin{equation}\label{eq:weight3}
\mu_{\Gcal, N, \beta, U}(w) : = \prod_{e \in \mathcal{E}} \frac{ \beta^{m_e(w)}}{m_e(w)!} \, \, \prod_{x \in \mathcal{V}} U_x(w) \qquad\qquad \forall w = (m, c, \pi) \in \mathcal{W}_{\mathcal{G}},
\end{equation}
where 
$U_x(w) := U \big ( \, \mathbf{n}_x(w), \mathbf{u}_x(w), t_x(w) \, \big )$
for each $x \in \mathcal{V}$.
Given a function
$f : \mathcal{W}_{\mathcal{G}} \rightarrow \mathbb{R}$, we use the same notation for the expectation of $f$,
$
\mu_{\Gcal,N, \beta, U}(f) : = \sum\limits_{w  \in \mathcal{W}_{\mathcal{G}}}  f(w) \, \, \mu_{\Gcal,N, \beta, U}(w).
$
\end{defn}
Thus, the weight function $U_x$ might depend not only on the total number of times $x$ is visited by paths of any given colour, but also on whether such paths are walks or loops.
Note that $\mu_{\Gcal,N, \beta, U}$ is not necessarily a probability measure and it does not necessarily have finite mass for all choices of $N \in \mathbb{N}_{>0}$ and $U$. 
In Section \ref{sect:special cases} we will prove 
that the random path model is equivalent to other 
models for certain choices of the parameters and the weight function, for which it is simple to deduce that the measure has finite mass, and hence so is for
$\mu_{\Gcal,N, \beta, U}$.
General sufficient conditions to ensure
that $\mu_{\Gcal,N, \beta, U}$ has finite mass can also be found in \cite[Proposition 3.1]{BenassiUeltschi}. 
Finally, note that the factorial term in the denominator of (\ref{eq:weight3}),
which was not present in the informal definition (\ref{eq:informaldefinition}), 
could be incorporated into the weight function. 

We now introduce one of the central quantities, the two-point function.
\begin{defn}\label{def:partition functions}
For any set $A \subset \mathcal{V}$,  
define 
$\mathcal{S}_A$ as the set of configurations $w \in \mathcal{W}_{\mathcal{G}}$ such that $u_x^1(w) = 1$
for any $x \in A$ and $u_z^1(w) = 0$ for any $z \in \mathcal{V} \setminus A$.
Moreover, for any vertex $x \in \mathcal{V}$,  define $\mathcal{R}_x$ as the set of configurations 
$w \in \mathcal{W}_{\mathcal{G}}$ such that $u_x^1(w) = 2$
and $u_z^1(w) = 0$ for any $z \in \mathcal{V} \setminus \{x\}$.
We define 
 $Z_{\Gcal,N, \beta, U}(A) : = \mu_{\Gcal,N, \beta, U}(  \mathcal{S}_A  )$.
By a slight abuse of notation, we write 
  $Z_{\Gcal,N, \beta, U}(x,y)$ when $A = \{x,y\}$ such that $x \neq y$, and we define
    $Z_{\Gcal,N, \beta, U}(x,x) :=
     \mu_{\Gcal,N, \beta, U}(  \mathcal{R}_x  )$
     for any $x \in \mathcal{V}$.
 Finally, we define the two-point functions, 
\begin{align*}
G_{\Gcal,N, \beta, U}(A)  & : = \frac{Z_{\Gcal,N, \beta, U}(A)}{Z_{\Gcal,N, \beta, U} ( \emptyset ) } \quad \forall A \subset \mathcal{V}, \\ 
G_{\Gcal,N, \beta, U}(x,y)  & : = \frac{Z_{\Gcal,N, \beta, U}(x,y)}{Z_{\Gcal,N, \beta, U} ( \emptyset ) } \quad \forall x, y \in \mathcal{V}.
\end{align*}
\end{defn}
We use the convention that if $\mu_{ \mathcal{G}, N, \beta, U}( \mathcal{S}_{\emptyset}) = 0$, then $Z_{\Gcal,N, \beta, U}(\emptyset) : = 1$. 
This way the two-point function is always well-defined.
 Sometimes, for a lighter notation, we will omit the sub-scripts and, in the case that our graph is the torus $\T_L = \mathbb{Z}^d / L \mathbb{Z}^d$, we will write $Z_{L,N, \beta, U} =Z_{\Gcal,N, \beta, U}$ and $G_{L,N, \beta, U} =G_{\Gcal,N, \beta, U}$.

The quantity $Z_{\mathcal{G}, N, \beta, U}(A)$,
can be viewed as a sum over realisations 
$w \in \mathcal{W}_{\mathcal{G}}$ weighted by (\ref{eq:weight3}) such that there is precisely one unoriented $1$-walk start (or end) point at each vertex $x \in A$, and no 
unoriented $1$-walk start (or end) point at each vertex $x \not\in A$. 
In some of the special cases considered in Section \ref{sect:special cases}, such walks of colour $1$ will interact with an arbitrary number of loops of colour $i \in \{1, \ldots, N\}$. This is the case for the loop $O(N)$ model and of the loop representation of the spin $O(N)$ model. In other special cases, they will interact not only with an arbitrary number of loops of colour $i \in \{1, \ldots, N\}$,
but also with an arbitrary number of walks consisting of only one edge and having colour $3$. This is the case of random lattice permutations and of the double-dimer model. 
Note that if $A$ contains an odd number of vertices, then necessarily 
$Z_{\Gcal,N, \beta, U}(A)  = 0$.

\subsection{Special cases}
\label{sect:special cases}
In this section we will show that, under some specific assumptions on the number of colours $N \in \mathbb{N}_{>0}$ and on the weight function $U$, the two-point function of the random path model corresponds to the two-point function of several classical models in statistical mechanics.
In all the models defined below,  the function $G_{\mathcal{G}, N, \beta, U}(x,y)$ is conjectured to exhibit different behaviours as the distance between $x$ and $y$ increases in the limit of large boxes in $\mathbb{Z}^d$,
as the parameters in the definition of the model and the dimension $d \geq 2$ vary. 
We refer to the papers cited in the introduction for conjectures and known facts.
Our  Theorem \ref{theo:monotonicity} below states a general site-monotonicity property which holds for the two-point function of all these models.
In all the cases considered below, we let 
$\mathcal{G} = (\mathcal{V}, \mathcal{E})$ be an undirected, simple, finite graph, where $\mathcal{V}$ denotes the vertex set and $\mathcal{E}$ denotes the set of undirected edges. We also denote $\mathbf{n}=(n_1,\dots,n_N)$ and $\mathbf{u}=(u_1,\dots,u_N)$ for $n_i, u_i\in\N$, $i=1,\dots,N$.

\paragraph{Spin $O(N)$ model.} 
To begin, we define the spin $O(N)$ model precisely. 
Fix an integer $N \in \mathbb{N}_{>0}$.
We denote by  ${\varphi} \in {(\mathbb{S}^{N-1})}^{\mathcal{V}}$
the spin configurations,
where $\mathbb{S}^{N-1}\subset\R^N$ is the unit sphere of dimension $N-1$. For example, $\mathbb{S}^{0} = \{-1,1\}$ and 
$\mathbb{S}^1 \subset \mathbb{R}^2$ is the unit circle. We will often write spin configurations as $\varphi=(\varphi_x)_{x\in\Vcal}$ where $\varphi_x = (\varphi_x^1, \ldots, \varphi_x^N) \in \mathbb{R}^N$ is the value of $\varphi$ at the vertex $x\in\Vcal$.
The hamiltonian of the spin $O(N)$ model is defined as 
\begin{equation}\label{eq:hamiltonian}
H_{\Gcal,N}(\varphi) = - \sum\limits_{ \{x,y\} \in \mathcal{E}}  \,    \varphi_x \cdot \varphi_y,
\end{equation}
where $\varphi_x \cdot \varphi_y$ denotes the usual inner product of two $N$-component vectors. For a parameter $\beta\geq 0$ known as the \emph{inverse temperature} the partition function at inverse temperature $\beta$ is given by 
\begin{equation}\label{eq:partition function}
Z_{\Gcal,N, \beta}^{spin}  = 
\Big  (  \,  \prod_{x \in \mathcal{V}} 
\int_{\mathbb{S}^{N-1}} d \varphi_x    \,   \Big  ) 
e^{- \beta H_{\Gcal,N}(\varphi)},
\end{equation}
where $d \varphi_x$ denotes the uniform 
probability measure on $\mathbb{S}^{N-1}$, that is,
$\int_{\mathbb{S}^{N-1}} d \varphi_x = 1$.
We introduce an expectation operator $\langle \cdot \rangle_{\mathcal{G}, N, \beta}$ on functions 
$(\mathbb{S}^{N-1})^{\mathcal{V}} \rightarrow \mathbb{R}$, that assigns the value
\begin{equation}\label{eq:deffunctional}
{\langle f \rangle}_{ \Gcal, N, \beta} = 
\frac{1}{Z_{ \Gcal,N, \beta}^{spin}} \,  
\Big  (  \,  \prod_{x \in \mathcal{V}} 
\int_{\mathbb{S}^{N-1}} d \varphi_x    \, \Big ) 
f \big ( \, (\varphi_x)_{x \in \mathcal{V}} \,  \big ) \, e^{- \beta H_{\Gcal,N}(\varphi)}.
\end{equation}
The next proposition formalises the correspondence between the correlation function of the classical spin $O(N)$ model and the point-to-point function of the random path model. 
Recall the definition of the two-point function, Definition \ref{def:partition functions}.
\begin{prop}\label{prop:equivalence}
Let $\mathcal{G} = ( \mathcal{V}, \mathcal{E})$ be a finite graph, fix an integer $N \in \mathbb{N}_{>0}$. When the weight function $U : \mathbb{N}^{2N+1} \rightarrow \mathbb{R}$ is defined as follows,
\begin{equation}\label{eq:weightSpin}
U( \mathbf{n}, \mathbf{u},t) = 
\begin{cases}
\frac{\Gamma(\frac{N}{2})}{2^n \, \Gamma( n + \frac{N}{2})}&
\quad  \mbox{if } u_2 = u_3 = \ldots u_N = 0, \\
0& \quad \mbox{if otherwise},
\end{cases}
\end{equation}
where $n =  n_1 + n_2 + \ldots + n_N$ and there is no dependence on $t$,
we have that, for any $\beta \in \mathbb{R}_{\geq 0}$, $A \subset \mathcal{V}$,
\begin{equation}\label{eq:conclusionequivalence}
\langle  \prod_{x \in A} \varphi^1_x  \rangle_{\mathcal{G}, N, \beta} =   G_{\Gcal,N, \beta, U}(A).
\end{equation}
\end{prop}
The choice of the weight function (\ref{eq:weightSpin}) is such that only realisations  $w \in \mathcal{W}_{\mathcal{G}}$ which present no walk of colour $i=2,\ldots N$ are allowed, and the weight depends on the total number of times a loop of arbitrary colour or a walk of colour $1$ visits the vertices.
Thus, by Definition \ref{def:partition functions}, the point-to-point function which appears in the right-hand side of (\ref{eq:conclusionequivalence})
corresponds to the ratio between the weight of realisations  $w \in \mathcal{W}_{\mathcal{G}}$ with an arbitrary number of loops and  precisely one walk of colour $1$ starting (or ending) at any vertex of $A$, and the weight of realisations $w \in \mathcal{W}_{\mathcal{G}}$ with only loops.
The proof of this proposition is presented in the appendix of this paper. It is a re-adaptation of \cite[Proposition 6.3]{BenassiUeltschi},
where a different correlation function was considered.

A classical case of study is the spin O(N) model in the presence of an external magnetic field, which corresponds to the case in which a term $h \sum_{x \in \mathcal{V}} \varphi_x^{N}$ is added to the right-hand side of (\ref{eq:hamiltonian}), where $h \in \mathbb{R}$ is non-zero. Proposition \ref{prop:equivalence} can be adapted to this case by modifying 
(\ref{eq:weightSpin}) appropriately and the monotonicity properties for the two point functions $< \varphi_x^i \,  \varphi_y^i  >_{\mathcal{G}, N, \beta}$ when   $i = 1, \ldots, N$,
Theorem \ref{theo:monotonicityspinO(N)},  hold with no change in the proof.

\paragraph{Loop $O(N)$ model.} 
The loop $O(N)$ model is defined as follows.
Given a set $A \subset \mathcal{V}$, we let  $\Omega_{loop}(A)$ be the set of spanning sub-graphs of  $\mathcal{G}$ such that the degree of every vertex  in $A$ equals one and the degree of every vertex in $\mathcal{V} \setminus A$ equals zero or two.
It follows from this definition that each realisation $\omega \in \Omega_{loop}(A)$ can be viewed as a collection of vertex-self-avoiding vertex-disjoint loops and walks, where at every vertex of $A$ one walk starts (or ends), and no walk starts or ends at the vertices of $\mathcal{V} \setminus A$.
Let $N \in \mathbb{N}_{>0}$  and  $\beta \in \mathbb{R}_{> 0}$ be two parameters.
We define 
$$
Z^{loop}_{\Gcal,N, \beta }(A) := \sum\limits_{\omega \in \Omega_{loop}(A)} \beta^{ e(\omega) } N^{ \ell(\omega)},
$$
where $e(\omega)$ is the total number of edges and $\ell(\omega)$ is the total number of loops of the graph $\omega \in \Omega_{loop}(A)$, and we define
$
G^{loop}_{\mathcal{G},  N, \beta}( A ) := Z^{loop}_{\Gcal,N, \beta }(A)/Z^{loop}_{\Gcal,N, \beta }(\emptyset).
$
Under the choice of the weight function of the random path model as follows, 
\begin{equation}\label{eq:weightoop}
U( \mathbf{n}, \mathbf{u},t) := 
\begin{cases}
1 \mbox{ if } n \in \{0, 1\}, \mbox{ and }  t=u_i = 0 \mbox{ for } i=2, \ldots N,   \\
0 \mbox{ if } otherwise,
\end{cases}
\end{equation}
where $ n =n_1 + \ldots + n_N$,
we have that, for any $A \subset \mathcal{V}$,
\begin{equation}
Z^{loop}_{\mathcal{G},  N, \beta}( A ) = Z_{\mathcal{G},  N, \beta, U}(A).
\end{equation}
This follows as the choice of $U$ ensures that we have an ensemble of vertex-self-avoiding, vertex-disjoint arbitrarily coloured loops and 1-walks, hence $m_e(w)=1$ for any $w\in\Wcal_{\Gcal}$ that contributes to $Z^{loop}_{\mathcal{G},  N, \beta}( A )$. From this it follows that
\begin{equation}\label{eq:correspondenceloopO(N)}
G^{loop}_{\mathcal{G},  N, \beta}( A ) = G_{\mathcal{G},  N, \beta, U}(A),
\end{equation}
where on the right-hand side we have the point-to-point function of the random path model.
The definition of the loop $O(N)$ model that we provided can be found for example in \cite{PeledSpinka} in the case of the hexagonal lattice. An alternative definition  where the loops are allowed to share the vertices but they are not allowed to share the edges is provided in \cite{Chayes}.
To obtain the loop $O(N)$ model which was defined in
\cite{Chayes}, one would need to define the weight function $U$ slightly differently than (\ref{eq:weightoop}).
See also \cite{DC1, DC2, GlazmanManolescu1, GlazmanManolescu2, Taggi} for recent papers.

\paragraph{Random lattice permutations.}
Let $\Omega_{per}(\emptyset)$ be the set of permutations $\pi : \mathcal{V} \rightarrow \mathcal{V}$   such that, for every $z \in \mathcal{V}$, either  $\{z, \pi(z)\} \in \mathcal{E}$ or $\pi(z) =z$. For any pair of distinct vertices $x, y \in \mathcal{V}$, let $\Omega_{per}(x,y)$ be the set of functions $\pi : \mathcal{V} \setminus \{y\}  \rightarrow \mathcal{V} \setminus \{x\}$ such that, for  every 
$z \in \mathcal{V}$, either $\{z, \pi(z)\} \in \mathcal{E}$ or $\pi(z) =z$, and, moreover,  every $z \in \mathcal{V} \setminus \{x,y\}$ has precisely one input and one output in $\pi$ (from this it also follows that $x$ has precisely one output and that $y$ has precisely one input).
This model has been studied in \cite{Betz, Betz2, Betz3}.

By drawing a directed edge from $z$ to $\pi(z)$ for every $z$ such that $\pi(z) \neq z$,  we see that any $\pi \in \Omega_{per}(\emptyset)$ can be viewed as a collection of vertex-disjoint oriented vertex-self-avoiding loops and dimers, where a dimer is a pair of nearest-neighbour edges, $z_1,z_2$, such that $\pi(z_1) =z_2$ and $\pi(z_2)=z_1$. Similarly, any $\pi \in \Omega_{per}(x,y)$ can be viewed as a collection of an arbitrary number of mutually-disjoint oriented self-avoiding loops and dimers and of one directed self-avoiding walk starting at $x$ and ending at $y$.
For simplicity we will only consider sets $A\subset \mathcal{V}$ of the form $\{x,y\}$ or $\emptyset$.
For any $\alpha \in \mathbb{R}$, we define,
$$
Z^{per}_{\Gcal, \alpha }(A) := \sum\limits_{\pi \in \Omega_{per}(A)} \exp \big ( - \alpha  \, e(\pi) \, \big ),
$$
where, here, $e(\pi) := \sum_{x \in \mathcal{V}} \mathbbm{1}\{x \neq \pi(x)\}$.
Moreover,  we define
$
G^{per}_{\mathcal{G}, \alpha} (A) :=  Z^{per}_{\Gcal, \alpha}(A) / Z^{per}_{\Gcal, \alpha}(\emptyset).
$
To represent random lattice permutations as a special case of the random path model, we need to fix $N=3$ and choose the weight function as follows,
\begin{equation}\label{eq:weightpermutations}
\quad U( \mathbf{n}, \mathbf{u},t) := 
\begin{cases}
1 \quad  & \mbox{ if } n \in \{0, 1\} \mbox{ and }  t=u_2 = u_3 = n_3 = 0, \\
\beta^{\frac{1}{2}}  \quad  & \mbox{ if } n  = 1 \mbox{ and }  u_3 = 1,   \\
0 & \mbox{ otherwise.}
\end{cases}
\end{equation}
where $ n =n_1 +  n_2 + n_3$.
Under this choice, we have that, for any $\alpha \in \mathbb{R}$, $\beta = e^{- \alpha}$ and $A = \{x,y\}$ for $x\neq y$ or $A=\emptyset$,
\begin{equation}
Z^{per}_{\Gcal, \alpha }(A) = Z_{\Gcal,3, \beta,  U}(A).
\end{equation}
The previous correspondence holds true since, by choosing $U$ as above, any $w \in \mathcal{W}_{\mathcal{G}, 3, \beta, U}$ 
can be viewed as a collection of walks of length two (dimers) of colour $3$, which are weighted by $\beta^2 = e^{- 2 \alpha}$ as in random permutations (one factor $\beta$ corresponds to the edge-weight and the additional factor $\beta$ corresponds to the product of the vertex-weights associated to the end-points of the dimer), undirected loops which have multiplicity two (as in random lattice permutations, where the loops are uncoloured but directed and for this reason they have multiplicity two as well) and a self-avoiding walk of colour $1$ connecting $x$ to $y$ which does not overlap with loops and dimers. It follows from this that
\begin{equation}\label{eq:correspondencepermut}
G^{per}_{\Gcal, \alpha }(A) = G_{\Gcal,3, \beta,  U}(A).
\end{equation}

\paragraph{Dimer model and double-dimer model.}
Given $\mathcal{G} = (\mathcal{V}, \mathcal{E})$ and $A \subset \mathcal{V}$, let $\mathcal{G} \setminus A$ be the sub-graph of $\mathcal{G}$ which is obtained from $\mathcal{G}$ by removing all vertices of $\mathcal{V}$ which are in $A$ and all the edges which are incident to the vertices in $A$.
A \textit{dimer cover} of a graph is a spanning sub-graph of that graph such that every vertex has degree precisely one.
Let $\Omega^{dim}(A)$ be the set of dimer covers of  the graph $\mathcal{G} \setminus A$
(which might be the empty set in some cases).
For the dimer model, we define
\begin{equation}\label{eq:pointdimer}
G_{\mathcal{G}}^{dim}(A) := \frac{ |\Omega^{dim}(A)|}{|\Omega^{dim}(\emptyset)|}.
\end{equation}
In the dimer model, the quantity 
$G_{\mathcal{G}}^{dim}(\{x,y\})$ is usually referred to as \textit{monomer-monomer correlation} and it is a central quantity in the  study of this system. Its value can be computed explicitly on some planar graphs $\mathcal{G}$
 \cite{Estelle, Kasteleyn, Kenyon3}.
As far as we know, on the torus of dimension $d \geq 3$,
our new general monotonicity property, Theorem \ref{theo:monotonicity} below,
is the only known fact on the behaviour of $G_{\mathcal{G}}^{dim}(\{x,y\})$.
To explain how (\ref{eq:pointdimer}) is obtained from the definition of the random path model, we introduce the double-dimer model.

A double-dimer configuration is the union of two dimer covers \cite{Kenyon} of $\mathcal{G}$.
More precisely, consider   $A \subset \mathcal{V}$ such that $A$ contains two distinct vertices, $A= \{x,y\}$, or $A = \emptyset$, and define
\begin{equation}\label{eq:definitionspacedd}
\Omega^{d.d.}(A) := \Omega^{dim}(A) \times \Omega^{dim}(\emptyset).
\end{equation}
Any realisation of the double-dimer model,  $\omega = (\omega_1, \omega_2) \in \Omega^{d.d.}(A)$,
can be viewed as a collection of an arbitrary number of mutually-disjoint self-avoiding loops and double-dimers
 and a self-avoiding walk connecting $x$ and $y$ by superimposing $\omega_1$ and $\omega_2$ 
 (a double-dimer corresponds to the superposition of two dimers occupying the same edge).
In the double dimer model, the measure on the configuration space (\ref{eq:definitionspacedd}) is the uniform measure.
For any set $A = \{x,y\}$, define
\begin{equation}\label{eq:connectiondimerdoubledimer}
G_{\mathcal{G}}^{d.d.}(A) := \frac{|\Omega^{d.d.}(A)|}{| \Omega^{d.d.}(\emptyset)|} = 
 \frac{|\Omega^{dim}(A)| |\Omega^{dim}(\emptyset)|}{| \Omega^{dim}(\emptyset)|^2} = 
 G_{\mathcal{G}}^{dim}(A),
\end{equation}
 and note that the second identity follows from (\ref{eq:definitionspacedd}),
while the third identity follows from the definition (\ref{eq:pointdimer}) after a simplification.
Recall that $n = n_1 + \ldots  n_N$ and choose the weight function $U$ as follows,
\begin{equation}\label{eq:weightdoubledimer}
\quad U( \mathbf{n}, \mathbf{u},t) := 
\begin{cases}
1 \quad  & \mbox{ if } n  = 1 \mbox{ and }  t=u_2 = u_3 = n_3 = 0, \\
1  \quad  & \mbox{ if } n  = 1 \mbox{ and }  u_3 = 1,   \\
0 & \mbox{ otherwise,}
\end{cases}
\end{equation}
and note that,  under this choice,
$$
| \Omega^{d.d.}(\{x,y\})    |= Z_{\mathcal{G}, 3, 1, U}(\{x, y\}), \quad \quad 
| \Omega^{d.d.}(\emptyset)    |= Z_{\mathcal{G}, 3, 1, U}(\emptyset).
$$
Indeed, the weight function $U$, allows mutually-disjoint self-avoiding loops having colour $1$ or $2$ and walks consisting of just one link which has colour $3$ and does not share his end-points with any other link, moreover for every vertex there exist either one or two links which are incident to it.
The value $Z_{\mathcal{G}, 3, 1, U}(\{x, y\})$ corresponds  to the number of such configurations with a `long' walk of colour $1$ connecting $x$ and $y$, while the value $Z_{\mathcal{G}, 3, 1, U}(\emptyset)$ corresponds to the number of such configurations with no `long' walk of colour $1$.
Thus,  each loop has multiplicity two, like the loops in the double-dimer model, and each single link has multiplicity one, like double-dimers in the double-dimer model. Moreover,  each vertex is touched by a loop or by a single link like in the double dimer model, in which every vertex belongs to a loop or to a double-dimer.
This explains the previous two identities. From such identities we deduce that,
$
G_{\mathcal{G}}^{d.d.}(\{x,y\}) = G_{\mathcal{G}, 3, 1, U}(x,y).
$
Thus, it follows from this relation and from (\ref{eq:connectiondimerdoubledimer}) that,
\begin{equation}\label{eq:connectiondimer}
G_{\mathcal{G}}^{dim}(\{x,y\}) = G_{\mathcal{G}, 3, 1, U}(x,y).
\end{equation}
This explains why our theorems  also apply to $G_{\mathcal{G}}^{dim}(\{x,y\})$.

\subsection{Main results}
We now state our main results. 
Our first two theorems generalise to the spin $O(N)$ model with arbitrary $N \in \mathbb{N}_{>0}$ and to all the models which were introduced in the previous section the site-monotonicity properties which were derived by Messager and Miracle-Sole,  \cite[Corollary 1]{MMS}  by using Lebowitz inequalities \cite{Lebowitz}.
In \cite{MMS},  the cases of free and `plus' boundary conditions in presence of a uniform external magnetic field  when  $N=1$  and of free boundary conditions with no external magnetic field when $N = 2$
  have been considered. 
An alternative derivation of such monotonicity properties has been given in \cite[Theorem 3.1]{Hegerfeldt} in the specific case of the Ising model, $N=1$.
Contrary to \cite{MMS}, the results of \cite{Hegerfeldt} also apply to periodic boundary conditions and they are not restricted to nearest neighbour interactions.
Note that, although the monotonicity properties which have been derived in  \cite{Hegerfeldt, MMS} hold for a much more restrictive class of models,  they are stronger than ours, since they do not require averaging  the two point function at two sites like in equation (\ref{eq:monotonicityforaveragecorrelation1}) below.

We use $\boldsymbol{e}_i$ to denote cartesian vectors
and $o = (0, 0, \ldots, 0) \in \T_L$ to denote the origin.
\begin{thm}[\textbf{Site-monotonicity for paths}]\label{theo:monotonicity}
Consider the torus $\T_L=\Z^d/L\Z^d$ in dimension $d\geq 2$ with $L\in2\N$ and nearest-neighbour edges and let $i\in\{1,\dots,d\}$. Assume that $U$ is defined as in Definition \ref{def:measure}, let $N \in \mathbb{N}_{>0}$ and $\beta \geq 0$ be arbitrary,
suppose that the measure $\mu_{(\T_L, \mathcal{E}_L), N, \beta, U}$ 
defined in Definition \ref{def:measure}
has finite mass.
If we write $z=(z_1,\dots,z_d)$ then, 
\begin{align}\label{eq:projectioninequality}
\mbox{if  $z_i\in 2\Z + 1$} \quad \quad G_{L, N, \beta, U}(o, z) & \leq G_{L, N, \beta, U}(o, z_i \be_i), \\
\label{eq:projectioninequality2}
\mbox{if  $z_i\in 2\Z \setminus \{0\}$} \quad  \quad  G_{L, N, \beta, U}(o, z)  & \leq \frac{1}{2} G_{L, N, \beta, U}(o, (z_i-1) \be_i) + 
\frac{1}{2} G_{L, N, \beta, U}(o, (z_i+1) \be_i).
\end{align}
Further, for $y\in\T_L$ with $y\cdot \be_i=0$ (possibly $y=o$) the function
\begin{equation}\label{eq:monotonicityforaveragecorrelation1}
G_{L, N, \beta, U}(o, y + n \be_i) + G_{L, N, \beta, U}(o, n \be_i), 
\end{equation}
is a non-increasing function of $n$ for odd $n$ in $(0,L/2)$.
\end{thm}
From the second statement of the previous theorem (applied when $y=o$) we deduce that 
$G_{L, N, \beta, U}(o,  n \be_i)$ is a non-increasing 
function of $n$ for odd $n$ in $(0,L/2)$.
Thus  we deduce that, for any $L \in 2 \mathbb{N}$, any coordinate $i \in \{1, \ldots, d\}$,
and any site $z = (z_1, z_2, \ldots, z_d)  \in \T_L$, 
\begin{equation}\label{eq:monotonicity2}
G_{L, N, \beta, U}(o,z) \leq 
G_{L, N, \beta, U}(o,z_i \boldsymbol{e}_i) \leq 
G_{L, N, \beta, U}(o,  \boldsymbol{e}_i).
\end{equation}
This theorem can be viewed as a statement about the geometry of  random (or self-avoiding) walks interacting with ensembles of loops and dimers.
When we force a long walk connecting two points $o$ and $z$ it is \textit{always} the case that the most favourable thing for the system in terms of energy-entropy balance is that such two points are neighbours, in such a way that the walk resembles the other objects (which are  loops, dimers, or both, depending on $U$).
From such monotonicity properties we also deduce that, in great generality, the two point function is not only bounded from above uniformly in the sites, but also in the  size of the torus (indeed, under the insertion of one link, which has a finite cost, one sees that the right-hand side of (\ref{eq:monotonicity2}) is a finite constant independent from $L$ times a probability).
The same fact would not hold true in the absence of any interaction: for example the two-point function between $o$ and $\boldsymbol{e}_1$ of a \textit{self-avoiding walk} weighted by $\beta$ \cite{MadrasSlade} (with no loops) diverges with $L$ when  $\beta$ is large enough. 

\begin{rem}
We deduce  from (\ref{eq:connectiondimer}) and from (\ref{eq:monotonicity2}) that, when $L \in 2 \mathbb{N}$,  the monomer-monomer correlation function of the dimer model, which was defined before (\ref{eq:connectiondimer}), satisfies
the following uniform upper bound,
\begin{equation}\label{eq:upperbounddimerts}
\forall z \in \T_L, \quad G_{L}^{dim}(\{o, z\}) \leq G_{L}^{dim}(\{o, \boldsymbol{e}_1\}) = \frac{1}{2d},
\end{equation}
where the identity follows from rotational symmetry and from the fact that, when the monomers are placed at $o$ and  $\boldsymbol{e}_1$, the monomer-monomer correlation function equals the probability that a uniform dimer cover with no monomers has a dimer on the edge $\{o, \boldsymbol{e}_1\}$.
This improves the upper bound on the ratio between the number of dimer covers with two monomers at arbitrary positions on the torus and the number of dimer covers with no monomers which was derived in \cite[Theorem 2]{KenyonC}.
There, it was proved through a multi-valued map principle that such a ratio, which is denoted by $\alpha$, is less or equal than $\frac{L^{2d}}{4}$ in any dimension $d \geq 3$. 
Using  translation invariance and the fact that $G_{L}^{dim}(\{o, z\})$ is zero whenever $z$ has  even graph distance from $o$, we deduce from (\ref{eq:upperbounddimerts}) that $\alpha \leq \frac{L^{2d}}{8d}$.
\end{rem}

\begin{rem}\label{remark:evensites!}
The fact that Theorem \ref{theo:monotonicity} holds only for `odd points' is not due to a limitation of our method. Indeed, one cannot expect that such a monotonicity property holds for any integer $n$ (odd or even) for any weight function $U$ satisfying the assumption of Definition \ref{def:measure}.
To explain this, recall the definition of the double-dimer model. 
Since no dimer cover of the graph $(\mathbb{Z}^2 \setminus L \mathbb{Z}^2) \setminus \{o, 2 \boldsymbol{e}_1 \}$
exists,
we deduce that $G_{L}^{d.d.}(o, 2 \boldsymbol{e}_1 )
= 0$.
Thus, if the previous theorem was true at any point (like Theorem \ref{theo:monotonicityspinO(N)} below), then  we would conclude that \emph{for any} 
even $L \in \mathbb{N}$ and any $z \in \T_L$ such that $\| z \| \geq 2$,
$G_{L}^{d.d.}(o, z) =0$, which is not true!
\end{rem}

From this remark we will infer that the double-dimer model is reflection positive for reflection through edges but not for reflection through sites.
Nevertheless for some weight functions, $U$, the monotonicity properties of the previous theorem hold true for any integer $n \in (0, L/2)$ (odd or even). This is the case of the spin $O(N)$ model with arbitrary $N \in \mathbb{N}_{>0}$, as the next theorem states.

\begin{thm}[\textbf{Site-monotonicity for spins}]\label{theo:monotonicityspinO(N)}
Consider the torus $\T_L=\Z^d/L\Z^d$ in dimension $d\geq 2$ with $L\in2\N$ and nearest-neighbour edges and let $i\in\{1,\dots,d\}$. 
Let $\langle \varphi_o \cdot \varphi_z \rangle _{L, N, \beta}$ be the spin-spin correlation of the spin O(N) model on a torus of side length $L \in \mathbb{N}$,  $N \in \mathbb{N}_{>0}$, and inverse temperature $\beta$.
We have that, for any $z  = (z_1, \ldots, z_d) \in \T_L$,
\begin{equation}
\langle \varphi_o \cdot \varphi_{z} \rangle_{L, N, \beta} \leq  \langle \varphi_o \cdot \varphi_{z_i \boldsymbol{e}_i} \rangle_{L, N, \beta}.
\end{equation}
Moreover, for $y\in\T_L$ with $y\cdot \be_i=0$ (possibly $y=o$) the function
\begin{align*}
& \langle \, \varphi_o \cdot \varphi_{ y + n \be_i} \rangle_{L, N, \beta} + \langle \varphi_o \cdot \varphi_{n \be_i} \rangle_{L, N, \beta}
\end{align*}
is a non-increasing function of $n$ for any integer $n$ in $(0,L/2]$.
\end{thm}
From the previous theorem we deduce that the two-point function of the spin O(N) model with arbitrary $N \in \mathbb{N}$ satisfies a relation which is analogous to (\ref{eq:monotonicity2}) 
at all sites (not just those which are `odd').


Our monotonicity properties imply that, if the lower bound  in (\ref{eq:liminfpositive})  is close enough to one, which in the specific case of the spin O(N) model holds for large enough $\beta$, then the two-point correlation function is point-wise positive, as the next theorem states.
\begin{thm}[\textbf{Point-wise uniform positivity}]\label{theo:pointwise}
Let $\langle \varphi_o \cdot \varphi_z \rangle _{L, N, \beta}$ be the spin-spin correlation of the spin O(N) model on a torus of side length $L \in \mathbb{N}$, with $N \in \mathbb{N}_{>0}$, and inverse temperature $\beta$.
Suppose that $d \geq 3$, where $d$ is the dimension of the torus. Then, there exists $\beta_0^{\prime} < \infty$ such that if $\beta > \beta_0^{\prime}$, there exists $C = C(\beta,d,N) >0$ and $L_0 < \infty$ such that for any even $L > L_0$,
\begin{equation}\label{eq:ourbound}
\langle  \varphi_o \cdot  \varphi_z \rangle_{L, N, \beta}  \geq C \quad \forall z\,:\, \|z\|_{\infty} < \tfrac{L}{8}.
\end{equation}
\end{thm}
Note that the constant $1/8$ which appears in the right-hand side of (\ref{eq:ourbound}) could be replaced by any other constant in $(0,\frac{1}{4})$.

\section{Reflection positivity}
\label{sec:RP}
The main purpose of this section is to introduce 
the method of reflection positivity. The reader is encouraged to read the notes of Biskup \cite{Biskup} or the book of Friedli and Velenik \cite[Chapter 10]{FriedliVelenik}
for an introduction on this method.
From now on the underlying graph  $\mathcal{G}$ will be a torus of side length $L$, 
$(\T_L, \mathcal{E}_L)$ in dimension $d \geq 2$,
with edges connecting nearest-neighbour vertices.
We identify $\T_L=\Z^d/L\Z^d$ with the set 
$[0,L)^d\cap\Z^d$.
We denote the \textit{origin},
corresponding to the vertex $(0, 0, \ldots, 0) \in \T_L$,  by $o \in \T_L$.
Throughout the section $N \in \mathbb{N}_{>0}$ and 
$\beta \in \mathbb{R}_{\geq 0}$ will be arbitrary but fixed, $U$ will be a weight function as in Definition \ref{def:measure},
 $L$ will be an even integer and we will assume that the measure 
$\mu_{L, N, \beta, U}$ which was defined in Definition \ref{def:measure}
has finite mass.

We now introduce the notion of domain and restriction and, after that, we introduce reflections. Intuitively, a function with domain $D \subset \mathcal{V}$ is a function which depends only on how $w \in \mathcal{W}_{\mathcal{G}}$ looks in $D$ or in a subset of $D$. More precisely, the function might only depend on how many links emanate from the vertices of $D$, on the direction in which they emanate, on which colour they have and on the pairings on vertices in $D$.


\paragraph{Domains.}
A function $f : \mathcal{W}_{\mathcal{G}}\to \R$ has \textit{domain} $D \subset \mathcal{V}$ if for any
pair of configurations $w = (m, c, \pi), w^{\prime} = (m^{\prime}, c^{\prime}, \pi^{\prime}) \in \mathcal{W}_{\mathcal{G}}$ such that 
$$
\quad \forall e \in \mathcal{E} :  \,e\cap D\neq \emptyset, \quad \forall z \in D, \quad m_e = m_e^{\prime} \quad c_e = c_e^{\prime} \quad \pi_z = \pi_z^{\prime} 
$$
one has that 
$f(w) = f(w^{\prime})$.


\paragraph{Restrictions.}
For $w=(m, c, \pi)\in\Wcal_{\Gcal}$ define the \emph{restriction} of $w$ to $D\subset \mathcal{V}$, $w_D=(m_D, c_D, \pi_D)$ with 
$c_D \in \mathcal{C}(m_D)$,
$\pi_D\in \Pcal_\Gcal(m_D, c_D)$, by 
  \vspace{-0.1cm}
 \begin{enumerate}[i)]
  \vspace{-0.1cm}
 \item $(m_D)_e = m_e$ for any edge $e \in \mathcal{E}$ which has at least one end-point in $D$ and $(m_D)_e =0$ otherwise,
  \vspace{-0.1cm}
  \item $(c_D)_e = c_e$ for any edge $e$ which has at least one end-point in $D$ and $(c_D)_e$ is the empty function otherwise,
 \vspace{-0.1cm}
  \item $(\pi_D)_x = \pi_x$ for any $x \in D$, and for $x\in \partial^e(D)$ we set $(\pi_D)_x$ as the pairing which leaves all links touching $x$ unpaired.
   \vspace{-0.1cm}
\end{enumerate}


\paragraph{Reflection through edges.}
Consider a plane $R$, which is orthogonal to one of the cartesian vectors $\boldsymbol{e_i}$,
$i \in \{1, \ldots, d\}$, and intersects the midpoint of
$L^{d-1}$ edges of the graph $(\T_L,\Ecal_L)$,  i.e.
$R = \{z \in \mathbb{R}^d \, \, : \, \, z \cdot \boldsymbol{e_i} = u      \}$, for some $u$ such that  $u - 1/2 \in \mathbb{Z} \cap [0,L)$
and $i \in \{1, \ldots, d\}$.
See Figure \ref{Fig:reflectionsexample} for an example. Given such 
a plane $R$, we denote by 
$\Theta : \T_L \rightarrow \T_L$  the reflection operator which reflects the vertices of  $\,\T_L$ with respect to $R$, i.e.
for any $x = (x_1, x_2, \ldots, x_d) \in \T_L$, 
\begin{equation}
\Theta(x)_k : = 
\begin{cases}
x_k & \mbox{ if } k \neq i, \\ 
2u - x_k   \mod  L & \mbox{ if } k = i.
\end{cases}
\end{equation}
Let $\T_L^+, \T_L^- \subset \T_L$ be the corresponding partition of the torus into two disjoint halves
such that $\Theta(\T_L^\pm) = \T_L^{\mp}$,
 as in Figure \ref{Fig:reflectionsexample}. 
Let $\Ecal^{+}_L, \Ecal^{-}_L \subset \Ecal_L$, be the set of edges $\{x,y\}$ with at least one of $x,y$ in $\T_L^+$ respectively $\T_L^-$. Moreover, let 
$\Ecal_L^R  := \Ecal^{+}_L \cap \Ecal^{-}_L$. Note that this set contains $2 L^{d-1}$ edges, half of them intersecting the plane $R$.
Further, let $\Theta  : \Wcal_\Gcal \rightarrow \Wcal_\Gcal$ denote the reflection operator reflecting the configuration $w=(m, c, \pi)$ with respect to $R$ (we commit an abuse of notation by using the same letter). More precisely we define $\Theta w=(\Theta m, \Theta c, \Theta \pi)$ where $(\Theta m)_{\{x,y\}}=m_{\{\Theta x,\Theta y\}}$, 
$(\Theta c)_{\{x,y\}}=c_{\{\Theta x,\Theta y\}}$, 
 $(\Theta \pi)_x=\pi_{\Theta x}$.
Given a function $f :\Wcal_{\Gcal}\to \mathbb{R}$, we also use the letter $\Theta$ to denote the reflection operator $\Theta$ which acts on $f$ as $\Theta f(w) : = f(\Theta w)$.
We denote by $\Acal^\pm$ the set of functions with domain $\T_L^\pm$ and denote by $\Wcal_\Gcal^\pm$ the set of configurations $w\in\Wcal_\Gcal$ that are obtained as a restriction of some $w^\prime\in\Wcal_\Gcal$ to $\T_L^\pm$. 


\paragraph{Projections.}
Finally, we denote by $\mathcal{W}_\Gcal^R$ the set of wire configurations $w = (m, c, \pi)$ 
such that  $m_e=0$ whenever $e\notin \Ecal_L^R$ and, for all $x\in\T_L$,
$\pi_x$ leaves all links touching $x$ unpaired.
We also denote by  $P_R:\Wcal_\Gcal\to \Wcal_\Gcal^R$ the projection such that, for any $w=(m,c, \pi)\in\Wcal_\Gcal$, 
$P_R(w) = (m^R, c^R, \pi^R)$ is defined as the wire configuration such that  $m^R_e = \mathbbm{1}_{e \in \mathcal{E}_L^R}m_e$ and $c^R_e =\mathbbm{1}_{e \in \mathcal{E}_L^R}  c_e$
and all links are unpaired at every vertex.

In the next proof we will use the following remark.

\begin{rem}\label{remark:observationidentif}
Recall the definition of restriction. Given a triplet of configurations $w^\prime \in \mathcal{W}_\mathcal{G}^R$, 
 $w_1 \in \mathcal{W}_{L}^{+}$,  $w_2 \in \mathcal{W}_{\mathcal{G}}^{-}$
such that  $P_R(w_1) = P_R(w_2) = w^\prime$,
there exists a unique configuration $w \in \mathcal{W}_\mathcal{G}$ such that 
$w_{\T_L^+} = w_1$, 
$w_{\T_L^-} = w_2$, 
$P_R(w) = w^\prime$.
This configuration is formed by concatenating $w_1$ and $w_2$
(concatenation includes  the pairing structures of each $w_j$).
\end{rem}

\begin{prop}\label{prop:RP1}
Consider the torus $(\T_{L},\Ecal_{L})$ for $L\in2\N$. Let $R$ be a reflection plane bisecting edges and  let $\Theta$ be the corresponding reflection operator. Consider the random path model with  $N \in \mathbb{N}_{>0}$,  inverse temperature $\beta \geq 0$ and weight function $U$ as in Definition \ref{def:measure}.
For any pair of functions
$f, g \in \mathcal{A}^+$, we have that,
\begin{enumerate}
\item[(1)] $\mu_{L,N, \beta, U} (f\Theta g)=\mu_{L,N, \beta, U} (g \Theta f)$,
\item[(2)] $\mu_{L,N, \beta, U} (f\Theta f)\geq 0$.
\end{enumerate}
From this we obtain that,
\begin{equation}\label{eq:RPCS}
\mu_{L,N, \beta, U} \big (   f \,  \Theta g  \big )
\leq 
\mu_{L,N, \beta, U} \big (   f  \,  \Theta f  \big )^{\frac{1}{2}}
\, \, 
\mu_{L,N, \beta, U} \big (   g \,  \Theta g  \big )^{\frac{1}{2}}.
\end{equation}
\end{prop}
\begin{proof}
Throughout the proof we will write $\mu=\mu_{L,N, \beta, U}$. First we note that \eqref{eq:RPCS} follows in the standard way from (1) and (2), since 
(1) and (2) show that we have a positive semi-definite, symmetric bilinear form. 

To prove (1) we note that, by Definition \ref{def:measure} and due to the symmetries of the torus, $\mu(w)=\mu(\Theta w)$ for any $w \in \mathcal{W}_\mathcal{G}$. Hence
\begin{equation}
\begin{aligned}
 \mu(f\Theta g)&=\sum_{w\in\Wcal_\Gcal}f(w)\Theta g(w)\mu(w)=\sum_{\Theta w\in\Wcal_\Gcal}f(\Theta w)\Theta g(\Theta w)\mu(w)
 \\
 &=\sum_{\Theta w\in\Wcal_\Gcal} g(w)\Theta f(w)\mu(w)=\sum_{ w\in\Wcal_\Gcal} g(w)\Theta f(w)\mu(w)=\mu(g\Theta f).
 \end{aligned}
 \end{equation}

For (2) we condition on the number of links in $w$ crossing the reflection plane and on their colours. 
We will  write 
 $m_{e}(w)$ for the number of links parallel to the edge $e$ of the configuration $w \in \mathcal{W}$.
We write
\begin{equation}\label{eq:decomposition}
\mu(f\Theta f)=\sum_{w \in\Wcal_\Gcal^R}\mu(f\,|\,w),
\end{equation}
where, for any $w^{\prime}   \in\Wcal_\Gcal^R$,
\begin{equation}
\begin{aligned}
\mu(f\,|\,w^{\prime}):=&\sum_{\substack{w\in\Wcal_\Gcal\\ P_R(w)=w^{\prime}}}f(w)\Theta f(w) \mu(w)
\\
=& \Big( \prod_{e\in\Ecal_L^R} \, \, \frac{m_e(\omega^{\prime})!}{\beta^{m_e(\omega^{\prime})}} \, \Big )  \, \,  \sum_{\substack{w\in\Wcal_\Gcal\\ P_R(w)=w^{\prime}}}f(w) \, 
\, \, \Big ( \prod_{e\in \Ecal_L^+}   \,\frac{\beta^{m_e(w)}}{m_e(w)!} \Big )\, \,  \Big (  \prod_{x\in\T_L^+}  \, U_x(w) \Big ) 
\\
&\qquad\qquad\qquad\qquad\qquad\qquad\qquad\Theta f(w) \, \, \Big (   \prod_{e\in\Ecal_L^-}\frac{\beta^{m_e(w)}}{m_e(w)!} \, \, \Big ) \, \, 
\Big (   \prod_{x\in\T_L^-} U_x(w)\Big ).
\end{aligned}
\end{equation}
Now, any $w\in\Wcal_\Gcal$ such that $P_R(w)=w^{\prime}$ uniquely defines $w_{\T_L^{\pm}}$, the restriction of $w$ to $\T_L^\pm$.
Thus, from Remark \ref{remark:observationidentif} we deduce that we can split the sum over $w\in\Wcal_\Gcal$ with $P_R(w)=w^{\prime}$ as the product of two independent sums and continue:
\begin{equation}
\begin{aligned}
\mu(f\,|\,w^{\prime})=& \Big ( \prod_{e\in\Ecal_L^R} \, \, \frac{m_e(\omega^{\prime})!}{\beta^{m_e(\omega^{\prime})}}  \Big ) \,  \bigg ( \sum_{\substack{w_1\in\Wcal_\Gcal^+\\ P_R(w_1)=w^{\prime}} }f(w_1) \, 
\, \, \Big ( \prod_{e\in \Ecal_L^+ }   \,\frac{\beta^{m_e(w_1)}}{m_e(w_1)!} \Big )\, \,  \Big (  \prod_{x\in\T_L^+}  \, U_x(w_1) \Big )  \bigg )
\\
&\qquad\qquad\qquad\qquad\qquad\qquad\qquad
\bigg ( \sum_{\substack{w_2\in\Wcal_\Gcal^-\\ P_R(w_2)=w^{\prime}} } \Theta f(w_2) \, 
\, \, \Big ( \prod_{e\in \Ecal_L^- }   \,\frac{\beta^{m_e(w_2)}}{m_e(w_2)!} \Big )\, \,  \Big (  \prod_{x\in\T_L^+}  \, U_x(w_2) \Big ) \bigg ) \\
=& 
\Big ( \prod_{e\in\Ecal_L^R} \, \, \frac{m_e(\omega^{\prime})!}{\beta^{m_e(\omega^{\prime})}}  \Big ) \,  
\bigg ( 
\sum_{\substack{w_1\in\Wcal_\Gcal^+\\ P_R(w_1)=w^{\prime}} }f(w_1) \, 
\, \, \Big ( \prod_{e\in \Ecal_L^+ }   \,\frac{\beta^{m_e(w_1)}}{m_e(w_1)!} \Big )\, \,  \Big (  \prod_{x\in\T_L^+}  \, U_x(w_1) \Big ) \bigg )^2 .
\end{aligned}
\end{equation}
The last equality holds true by the  symmetry of the torus. Since the last expression is non-negative, from (\ref{eq:decomposition}) we conclude the proof of (2) and, thus, the proof of the proposition.
\end{proof}

Recall that $u^i_x(w)$ denotes the number of $i$-links touching $x \in \T_L$ which are not paired to any other link at $x$ (i.e. the number of walks with colour $i$ with an end-point at $x$). 
For any 
field
$\boldsymbol{h} = (h_x)_{x \in \T_L}\in \R^{\T_L}$, we define the very important quantity,
\begin{equation}
Z^{\text{field}}_{L,N, \beta, U}( \boldsymbol{h} ) : = \mu_{L,N, \beta, U} \bigg (   \prod_{x \in \T_L} h_x^{u^1_x(w)}  \bigg ),
\end{equation}
using the convention that $0^0=1$.
We will always assume that the weight function  $U$ is such that  $Z^{\text{field}}_{L,N, \beta, U}( \boldsymbol{h} )$ is finite for any  real  vector $\boldsymbol{h} = (h_x)_{x \in \T_L}$ such that $h_x \leq 1$
for any $x \in \T_L$.
This assumption is certainly fulfilled in all the special cases considered above. 
Note that $Z^{\text{field}}_{L,N, \beta, U}(\boldsymbol{0})=Z_{L,N,\beta,U}(\emptyset)$ only involves configurations with no walks of colour 1.
For $\bh= (h_x)_{x \in \T_L}\in \R^{\T_L}$ we define its reflection $\Theta \bh=((\Theta h)_x)_{x \in \T_L}$ by $(\Theta h)_x=h_{\Theta x}$. 
We also define the related fields, $\bh^\pm=(h^\pm_x)_{x \in \T_L}$ by
\begin{equation}
h^{\pm}_{ {x}} =  
 \begin{cases}
{h}_{ {x}}   & \, \, \, \mbox{if} \,  \, x \in \T_L^{\pm},  \\
  {h}_{ {\Theta x}}  & \, \, \,  \mbox{if}   \,  \,  x \in \T_L^{\mp}. \\
\end{cases}
\end{equation}

\begin{prop}\label{prop:reflection positivity 2}
Under the same assumptions as Proposition \ref{prop:RP1}, for any field $\boldsymbol{h} \in \R^{\T_L}$ such that $Z^{\text{field}}_{L,N, \beta, U}( \boldsymbol{h} )$ is finite, we have that,
$$
Z^{\text{field}}_{L,N, \beta, U}( \boldsymbol{h} )
\leq Z^{\text{field}}_{L,N, \beta, U}( \boldsymbol{h}^+ )^{\frac{1}{2}} 
\, \, Z^{\text{field}}_{L,N, \beta, U}( \boldsymbol{h}^- )^{\frac{1}{2}}.
$$
\end{prop}
\begin{proof}
This is an application of Proposition \ref{prop:RP1}. Define $f_{\bh}(w):=   \prod_{x \in \T_{L}} h_x^{u^1_x(w)}$, then we can write $f_{\bh}(w)=f_{\bh}^+(w)f_{\bh}^-(w)$ where $f_{\bh}^\pm(w):=   \prod_{x \in \T_{L}^\pm} h_x^{u^1_x(w)}$. We have that $f_{\bh}^\pm\in\Acal^\pm$ as the function only depends on the pairings at sites in $\T_{L}^\pm$ and hence has domain $\T_{L}^\pm$. We have
\begin{equation}
f_{\bh}^\pm(w)\Theta f_{\bh}^\pm(w)= f_{\bh}^\pm(w) f_{\Theta \bh}^\mp(w)=\prod_{x\in\T_{L}}(h^\pm_x)^{u^1_x(w)}.
\end{equation}
Now applying Proposition \ref{prop:RP1} with $f=f_{\bh}^+$ and $g=\Theta f_{\bh}^-$ completes the proof.
\end{proof}

The central object of this section is the next weighted sum of two-point functions.
For any field of real values, $\boldsymbol{h} = (h_x)_{x \in \T_L}$, define
\begin{equation}\label{eq:Z2}
{Z}^{(2)}_{L,N, \beta, U}( \boldsymbol{h} ) : =
\sum_{\substack{x \in\T_{L}}}Z_{L,N, \beta,U}(x,y) \, h_x^2 \, \, + \, \, \frac{1}{2} \, \, 
\sum_{\substack{x,y \in\T_{L} : \\ x \neq y}}Z_{L,N, \beta,U}(x,y)h_xh_y,
\end{equation}
as a sum over weights of configurations with a single walk of colour 1 (and an arbitrary number of  walks of colour $2,\dots n$, $U$ allowing) weighted by the products of values of $\bh$ at the end points of the $1$-walk. 
Note that the second sum is in the right-hand side of (\ref{eq:Z2}) is over ordered pairs of distinct sites.  We have the following theorem. 

\begin{thm}\label{thm:RPaverage}
Under the same assumptions as Proposition \ref{prop:reflection positivity 2}, given an arbitrary vector field $\boldsymbol{h}$,
we have that,
\begin{align}
\label{eq:inequality for domination for one path}
{Z}^{(2)}_{L,N, \beta, U}( \boldsymbol{h} )  & \leq \frac{ {Z}^{(2)}_{L,N, \beta, U} (\bh^+) +  {Z}^{(2)}_{L,N, \beta, U} (\bh^-)   }{2}.
\end{align}
  \end{thm}

\begin{proof}
Throughout the proof we will write $Z^*_{L,N,\beta,U}=Z^*$ where $*\in\{\text{field},(2)\}$ for a lighter notation. Take $\bh\in\R^{\T_{L}}$ and $\eta>0$, we can expand $Z^{\text{field}}(\eta \boldsymbol{h} )$ into a series of terms in $\eta$, 
as follows
\begin{equation}
Z^{\text{field}}(\eta \boldsymbol{h} )=Z^{\text{field}}(\b0)+\eta^2 {Z}^{(2)}(\bh)+R_{\bh}(\eta),
\end{equation}
where $R_{\boldsymbol{h}}(\eta)$ is of order $O(\eta^4)$. $Z^{\text{field}}(\b0)$ involves weights of configurations with no walk of colour $1$.
The second term involves the weight of configurations with one unique walk of colour $1$. More precisely, the second term is the weight of configurations with precisely two distinct points  $x,y \in \T_L$ such that $u_x^1(w) = u_y^1(w) = 1$ and $u_z^1(w)=0$ for $z\in\T_L\setminus{\{x,y\}}$ or with one  point $x \in \T_L$  such that $u^1_x(w) = 2$ and $u_z^1(w)=0$ for $z\in\T_L\setminus{\{x\}}$.
Now using Proposition \ref{prop:reflection positivity 2} and the Taylor expansion ${(1+x)}^{1/2}=1 +x/2+O(x^2)$, we obtain that,
\begin{equation}
\begin{aligned}
&Z^{\text{field}}(\b0)+\eta^2{Z}^{(2)}(\bh)+R_{\bh}(\eta)
\\
&\leq \bigg[\big(Z^{\text{field}}(\b0)+\eta^2{Z}^{(2)}(\bh^+)+R_{\bh^+}(\eta)\big)\big(Z^{\text{field}}(\b0)+\eta^2{Z}^{(2)}(\bh^-)+R_{\bh^-}(\eta)\big)\bigg]^{\tfrac12}
\\
&= \bigg[\big(Z^{\text{field}}(\b0)^2+\eta^2Z^{\text{field}}(\b0)\big({Z}^{(2)}(\bh^+)+{Z}^{(2)}(\bh^-)\big)+O(\eta^4)\big)\bigg]^{\tfrac12}
\\
&=Z^{\text{field}}(\b0)+\eta^2\frac{{Z}^{(2)}(\bh^+)+{Z}^{(2)}(\bh^-)}{2}+O(\eta^4).
\end{aligned}
\end{equation}
As this inequality holds for arbitrarily small $\eta>0$ we see, by taking $\eta$ sufficiently small, that (\ref{eq:inequality for domination for one path}) holds, thus concluding the proof.
\end{proof}

\section{Proof of Theorems  \ref{theo:monotonicity},
\ref{theo:monotonicityspinO(N)}, and \ref{theo:pointwise}}
\label{sec:MonAndProofs}
The main purpose of this section is to prove Theorems \ref{theo:monotonicity}, \ref{theo:monotonicityspinO(N)}, 
and \ref{theo:pointwise}.
This section is divided into two subsections.
In the first subsection we will use the tools which have been introduced in Section \ref{sec:RP}
to prove  Theorem  \ref{theo:monotonicity},
which involves the random path model with arbitrary weight function.
In the second subsection we will consider only the spin $O(N)$ model and we will introduce a new type of reflection, \textit{reflection through sites},
which is not available for arbitrary weight functions but is in (at least) the case of the spin O(N) model.
Through this section we will always take our graph $\mathcal{G}$ to be the torus $\mathbb{Z}^d / L \mathbb{Z}^d$ with $d \geq 2$.
We fix $N \in \mathbb{N}_{>0}$, $\beta \geq 0$ and take $U$ such that the measure $\mu_{L, N, \beta, U}$, which was defined in Definition \ref{def:measure}, has finite mass for any $L \in 2 \mathbb{N}$.
For a lighter notation we will suppress some indices of our quantities of interest, keeping only their dependence on $L$. Also, for any $z \in  \T_L$, we will use the notation,
\begin{equation}
Z_L(z) := Z_L(o,z), \quad G_L(z) := G_L(o,z),
\end{equation}
(recall Definition \ref{def:partition functions}).

\subsection{Monotonicity for paths using reflection through edges and proof of Theorem \ref{theo:monotonicity}}

The next proposition is a consequence of Theorem \ref{thm:RPaverage} and states a convexity property for the  partition function $Z_L(z)$ for sites $z$ belonging to the cartesian axes.

\begin{prop}\label{prop:general sufficient}
Let $L \in 2 \mathbb{N}$,  let $z \in \T_L$ 
and let $\boldsymbol{e}_i$ be a cartesian vector. The following inequality holds for any integer $q \in \mathbb{N}$ such that  $q + z\cdot\be_i$ is odd and such that $z\cdot\be_i-q, z\cdot\be_i+q \in (0, L)$,
\begin{equation}
\label{eq:firstrelation}
Z_{L}( z)  \leq \frac{1}{2}\bigg( Z_{L}\big( (z\cdot\be_i-q) \, \boldsymbol{e_i}  \big ) +  Z_{L}\big (  (z\cdot\be_i+q ) \, \boldsymbol{e_i}  \big )\bigg).
\end{equation}
\end{prop}

\begin{proof}
Consider the field 
 $\boldsymbol{h} = (h_{x})_{x \in\T_{L}}$ given by
\begin{equation}\label{eq:hsinglewalk}
h_{x} = 
\begin{cases}
1 \mbox{ if } x \in \{ o, z\}   \\
0 \mbox{ otherwise} .
\end{cases}
\end{equation}
This means $\bh$ is zero except at two vertices,
which are represented by a square on the top of Figure \ref{Fig:RefPos}-left.
\begin{figure}[t]
\includegraphics[scale=1.2]{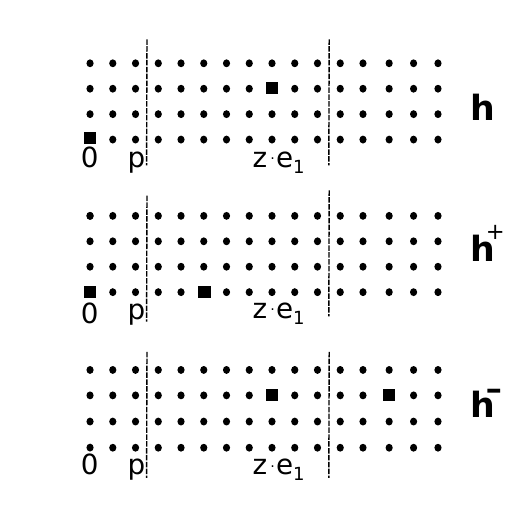}
\includegraphics[scale=1.2]{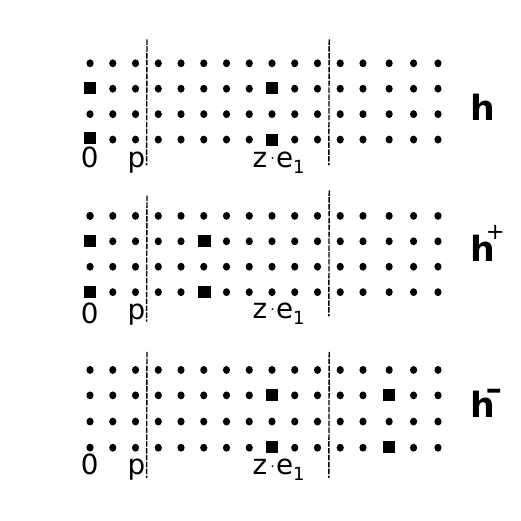}
\centering
\caption{We represent three slices of a torus of side length $L = 16$, the dashed lines separate $\T_L^+$ from $\T_L^-$. The vector fields equal 1 on the square vertices and 0 on the round vertices. \textit{Left:} the field $\boldsymbol{h}$ is chosen as in the proof of Proposition \ref{prop:general sufficient}. \textit{Right:} the field $\boldsymbol{h}$ is chosen as in the proof of Proposition \ref{prop:superharmonicity}.} 
\label{Fig:RefPos}
\end{figure}
Let $R$ be the reflection plane which is orthogonal to the vector $\boldsymbol{e_i}$ and which crosses the midpoint of the edge $\{p \,  \boldsymbol{e_i}, \, (p+1)  \boldsymbol{e_i}\}$, where
$$
p : = \frac{1}{2}(z\cdot\be_i-1+q),
$$
which is an integer since we assumed that $z\cdot\be_i+q$ is odd.
Moreover, since we assumed that 
 $z\cdot\be_i-q, z\cdot\be_i+q \in (0, L)$,
 we deduce that  $p$ satisfies $0\leq p<z\cdot \be_i < L$.
Thus, when we perform a reflection with respect to $R$, we obtain two fields $\boldsymbol{h^+}$ and $\boldsymbol{h^-}$
such that 
\begin{align}
h^+_{x} &  = 
\begin{cases}
1 \mbox{ if } x\in\{o,(2p+1)\be_i\}, \\
0 \mbox{ otherwise},
\end{cases} \\
h^-_{x} & = 
\begin{cases}
1 \mbox{ if } x\in\{z, z+(2p+1-2z\cdot\be_i)\be_i \}, \\
0 \mbox{ otherwise}.
\end{cases}
\end{align}
Note that the condition $0\leq p<z\cdot \be_i < L$ ensures that $\bh^+$ and $\bh^-$ are each non-zero at only two vertices.
For a representation of $\boldsymbol{h}$ and of the reflected fields see Figure \ref{Fig:RefPos}-left.
 Thus,  from translation invariance, reflection invariance and the definition (\ref{eq:Z2}), we deduce that,
\begin{align}\label{eq:cond1re}
 Z^{(2)}_{L}(\boldsymbol{h}) = &2 Z_{L}(o)+
 Z_{L}(z),
\\
\label{eq:cond2re}
Z^{(2)}_{L}(\boldsymbol{h^+}) = & 2 Z_{L}(o)+
 Z_{L}\big ( (2p+1)\be_i \big ),
\\
\label{eq:cond3re}
Z^{(2)}_{L}(\boldsymbol{h^-}) = & 2 Z_{L}(o)+
 Z_{L}\big ((2p+1-2z\cdot\be_i)\be_i \big ).
\end{align}
By applying Theorem \ref {thm:RPaverage},
 we conclude the proof after a cancellation of the terms 
$2 Z_{L}(o)$.
\end{proof}

The inequalities \eqref{eq:projectioninequality} and 
 \eqref{eq:projectioninequality2} in Theorem \ref{theo:monotonicity} are an immediate consequence of the previous proposition.
\begin{proof}[\textbf{Proof of \eqref{eq:projectioninequality} and \eqref{eq:projectioninequality2} in Theorem \ref{theo:monotonicity}}]
For \eqref{eq:projectioninequality}, we apply 
 Proposition \ref{prop:general sufficient} when $q=0$ and  $z \cdot \boldsymbol{e}_i$ is odd and then divide by $Z_L$.
 For \eqref{eq:projectioninequality2}, we apply 
 Proposition \ref{prop:general sufficient} when $q=1$ and  $z \cdot \boldsymbol{e}_i$ is even and positive and then divide by $Z_L$. 
 \end{proof}

For the remainder of this section we will work with the following sum of two point functions. For any $z \in \T_L$, and any unit vector $\boldsymbol{e}_i \in \mathbb{Z}^d$, we define the \textit{averaged two-point function},
\begin{equation}\label{eq:projectedtwopoint}
G_{L}^{\boldsymbol{e}_i}( z) := \frac{G_{L} (z  ) + G_{L}\big ( (z \cdot \boldsymbol{e}_i) \,  \boldsymbol{e_i} \,\big  )}{2}.
\end{equation}
In other words, given a point $z \in \T_L$ and a unit vector $\boldsymbol{e}_i \in \mathbb{Z}^d$, we average $G_L(z)$ with the value of $G_L$ evaluated  at the projection of $z$ 
onto the cartesian axis corresponding to $\boldsymbol{e}_i$. The reason why 
we introduce the averaged two-point function is that it satisfies a very useful monotonicity property. We remark that, if $z$ lies on the cartesian axis corresponding to $\boldsymbol{e}_i$, then the averaged two-point function $G_{L}^{\boldsymbol{e}_i}$ equals the two-point function, i.e,
$$
\forall k \in [0, L],  \quad G_{L}^{\boldsymbol{e}_i}( k \, \boldsymbol{e}_i) = G_{L}( k \,  \boldsymbol{e}_i).
$$
This means that, in this special case, the next statements also hold for the (non-averaged) two-point function.
The next proposition, applied with $q=2$, will lead to the monotonicity property of the averaged two-point function.
\begin{prop}\label{prop:superharmonicity}
For any $L \in 2 \mathbb{N}$, $q\in\N$ and $z \in \T_{L}$, such that 
$z \cdot \boldsymbol{e_i}+q$ is odd and
$z \cdot \boldsymbol{e_i}-q, z \cdot \boldsymbol{e_i}+q \in (0, L )$, the following inequality holds,
\begin{equation}\label{eq:super-harmonicity}
G_{L}^{\boldsymbol{e}_i} \big (  z + q \,  \boldsymbol{e_i} \big )  - 
G_{L}^{\boldsymbol{e}_i} \big (  z \big )
\geq 
G_{L}^{\boldsymbol{e}_i} \big ( z  \big )  - 
G_{L}^{\boldsymbol{e}_i}\big (  z - q \, \boldsymbol{e}_i \big ).
\end{equation}
\end{prop}
\begin{proof}
When $z \in \T_L$ is such that $z = (z  \cdot \boldsymbol{e}_i) \boldsymbol{e}_i$,
then the claim follows re-arranging the terms in Proposition \ref{prop:general sufficient} and dividing by $Z_L(\emptyset)$. Consider now a vertex $z \in \T_L$ satisfying our assumptions and not lying on the Cartesian axis $\boldsymbol{e}_i$.
We will prove that, under the assumptions of the proposition,
\begin{equation}\label{eq:goalpropposition}
Z_{L}((z\cdot \be_i)\be_i)+Z_{L}(z) \leq 
 \frac{1}{2}\big( Z_{L}((z\cdot\be_i+q)\be_i ) 
 +Z_{L}((z\cdot\be_i-q)\be_i)+Z_{L}(z+q\be_i)+Z_{L}(z-q\be_i)\big),
\end{equation}
from which (\ref{eq:super-harmonicity}) follows after dividing by $Z_L(\emptyset)$ and rearranging the terms (recall the definition of the two-point function which was introduced in Definition \ref{def:partition functions}).
As in Proposition \ref{prop:general sufficient} we need to make an appropriate choice of $\bh$. We choose the following field,
$$
h_{x} = 
\begin{cases}
1 \mbox{ if } x \in \{ o, z, (z\cdot\be_i)\be_i, z-(z\cdot\be_i)\be_i\},  \\
0 \mbox{ otherwise},
\end{cases}
$$
which is represented in Figure \ref{Fig:RefPos}-right.
As in Proposition \ref{prop:general sufficient}, define 
$$p: =\tfrac12 (z\cdot\be_i-1+q),$$
which is an integer since we assumed that 
$z \cdot \boldsymbol{e_i}+q$ is odd.
Let $R$ be the reflection plane which is orthogonal to the vector $\boldsymbol{e_i}$ and which crosses the midpoint of the edge $\{p \,  \boldsymbol{e_i}, \, (p+1)  \boldsymbol{e_i}\}$.
When we perform a reflection with respect to $R$, we obtain the fields $\boldsymbol{h^+}$ and $\boldsymbol{h^-}$
such that 
\begin{align}
h^+_{x} &  = 
\begin{cases}
1 \mbox{ if } x\in\{o,(2p+1)\be_i,z-(z\cdot \be_i)\be_i,z+(2p+1-z\cdot \be_i)\be_i\},  \\
0 \mbox{ otherwise},
\end{cases} \\
h^-_{x} & = 
\begin{cases}
1 \mbox{ if  } x\in\{z, (z\cdot\be_i)\be_i,(2p+1-z\cdot\be_i)\be_i,z+(2p+1-2z\cdot\be_i)\be_i\}, \\
0 \mbox{ otherwise}.
\end{cases}
\end{align}
Indeed, by our assumptions we deduce that
$0\leq p<z\cdot \be_i < L$ and this ensures that $\bh^+$ and $\bh^-$ are each non-zero at only four vertices.
See Figure \ref{Fig:RefPos}-right for an illustration of these fields.
Using translation and reflection invariance, we have that,
\begin{equation}\label{eq:fieldexpansions}
\begin{aligned}
Z^{(2)}_{L}(\bh)=& 4Z_{L}(o) + 2Z_{L}((z\cdot\be_i)\be_i)
+2Z_{L}(z-(z\cdot\be_i)\be_i) + 2Z_{L}(z),
\\
Z^{(2)}_{L}(\bh^+)=& 4Z_{L}(o) +2Z_{L}((2p+1)\be_i)
+2Z_{L}(z-(z\cdot\be_i)\be_i)+2Z_{L}(z+(2p+1-z\cdot\be_i)\be_i),
\\
Z^{(2)}_{L}(\bh^-)=& 4Z_{L}(o) +2Z_{L}((2p+1-2z\cdot\be_i)\be_i)
+2Z_{L}(z-(z\cdot\be_i)\be_i)+2Z_{L}(z-(2p+1-z\cdot\be_i)\be_i).
\end{aligned}
\end{equation}
Now using Theorem \ref{thm:RPaverage} we obtain  (\ref{eq:goalpropposition}) and this concludes the proof.
\end{proof}

An important consequence of this proposition is a proof of the second statement of Theorem \ref{theo:monotonicity}. Indeed, Proposition \ref{prop:superharmonicity} establishes a form of convexity for  $G_L^{\boldsymbol{e}_i}(y + n \boldsymbol{e}_i)$ as a function of $n$ when $y \cdot \boldsymbol{e}_i = 0$ and $n$ is an odd integer in $(0, L)$. In addition, this function is symmetric around $L/2$. Hence, it has to be non-increasing up to $L/2$ and non-decreasing afterwards.

\begin{proof}[\textbf{Proof of \eqref{eq:monotonicityforaveragecorrelation1} in Theorem  \ref{theo:monotonicity}}]
The proof is by contradiction.
Thus, suppose that there exists an odd integer $n \in (0, L/2)$
such that $G^{\boldsymbol{e}_i}_L(y + n \boldsymbol{e}_i) > G^{\boldsymbol{e}_i}_L(y + (n-2) \boldsymbol{e}_i)$. 
From, this assumption and from an iterative application of Proposition \ref{prop:superharmonicity} with $q=2$ we deduce that,
\begin{equation*}
\forall m \in (n,  L) \cap (2\mathbb{Z}+1)  \quad 
G^{\boldsymbol{e}_i}_L(y + (m+2) \boldsymbol{e}_i)-G^{\boldsymbol{e}_i}_L(y + m \boldsymbol{e}_i) \geq G^{\boldsymbol{e}_i}_L(y + n \boldsymbol{e}_i) -G^{\boldsymbol{e}_i}_L(y + (n-2) \boldsymbol{e}_i)>0,
\end{equation*}
which in particular implies that
$$
\forall m \in [n-2,  L) \cap (2\mathbb{Z}+1) \quad \quad  G^{\boldsymbol{e}_i}_L \big (y + (m+2) \boldsymbol{e}_i)>G^{\boldsymbol{e}_i}_L(y + m \boldsymbol{e}_i \big ),
$$
i.e, the function $G^{\boldsymbol{e}_i}_L(y + m \boldsymbol{e}_i \big )$ is strictly increasing  with respect to the odd integers in  $[n-2, L)$.
From this we deduce that,
\begin{equation*}\label{eq:tocontradict}
 G^{\boldsymbol{e}_i}_L\big (y -  n \boldsymbol{e}_i \big ) =  G^{\boldsymbol{e}_i}_L \big (y  + n \boldsymbol{e}_i +  (L - 2 n) \boldsymbol{e}_i   \big ) 
  >
  G^{\boldsymbol{e}_i}_L\big (y + n \boldsymbol{e}_i \big ).
\end{equation*}
However, the previous relation cannot hold true  by torus symmetry, thus we found the desired contradiction and concluded the proof.
\end{proof}

\subsection{Monotonicity for spins using reflection through sites and proof of Theorem \ref{theo:monotonicityspinO(N)}}
\label{subsect:sitesreflection}
In this section we define reflections in planes of sites. It is a classical fact that the Gibbs measure associated to the spin O(N) model is positive under reflections through sites.
We use such a notion of reflection positivity to 
obtain inequalities which are analogous to those which were proved in the previous section, but which hold at `even' points of the torus.

\paragraph{Reflections through sites.}
Consider a plane, $R$, which is orthogonal to one of the cartesian vectors $\boldsymbol{e_i}$,
$i \in \{1, \ldots, d\}$, and intersects 
$L^{d-1}$ sites of the graph $(\T_L,\Ecal_L)$,  i.e.
$R = \{z \in \mathbb{R}^d \, \, : \, \, z \cdot \boldsymbol{e_i} = m      \}$, for some  $m \in \mathbb{Z} \cap [0,L)$
and $i \in \{1, \ldots, d\}$.
See Figure \ref{Fig:reflectionsthroughsites} for an example. 
Given such 
a plane $R$, we denote by 
$\Theta : \T_L \rightarrow \T_L$  the reflection operator which reflects the vertices of $\T_L$ with respect to $R$, i.e.
for any $x = (x_1, x_2, \ldots, x_d) \in \T_L$, 
\begin{equation}
\Theta(x)_k : = 
\begin{cases}
x_k & \mbox{ if } k \neq i, \\ 
2m - x_k   \mod  L & \mbox{ if } k = i.
\end{cases}
\end{equation}
Let $\T_L^+, \T_L^- \subset \T_L$ be the corresponding decomposition of the torus into two (overlapping) halves ($\T_L^+\cup\T_L^-=\T_L$)
such that $\Theta(\T_L^\pm) = \T_L^{\mp}$.
We define
$\T_L^R := \T_L^+\cap \T_L^-$, which has cardinality $2 \, L^{d-1}$. We further define $\Ecal^{+}_L, \Ecal^{-}_L \subset \Ecal_L$ to be the set of edges $\{x,y\}$ with both $x$ and $y$ in $\T_L^+$ respectively $\T_L^-$.
Contrary to reflections through edges, we have 
$\Ecal^{+}_L \cap \Ecal^{-}_L = \emptyset$.
A reflection $\Theta$ in a plane $R$ through sites acts on functions $f:(\mathbb{S}^{N-1})^{\T_L}\to\R$ as $\Theta f(\varphi)=f(\Theta(\varphi))$ where $\Theta(\varphi)_x=\varphi_{\Theta(x)}$.
Let $\Bcal^{\pm}$ be the set of bounded measurable functions $f:(\mathbb{S}^{N-1})^{\T_L}\to\R$ depending only on spins in $\T_L^{\pm}$. More precisely, $f\in\Bcal^{\pm}$ if for any $\varphi,\varphi^{\prime}\in(\mathbb{S}^{N-1})^{\T_L}$ such that $\varphi_x=\varphi^{\prime}_x$ for all $x\in\T_L^\pm$ we have $f(\varphi)=f(\varphi^{\prime})$.

The next proposition is classical. For a proof, see for example \cite[Chapter 10]{FriedliVelenik}.
\begin{prop}\label{prop:RPsites}
Consider the torus $(\T_{L},\Ecal_{L})$ for $L\in2\N$. Let $R$ be a reflection plane bisecting vertices and  let $\Theta$ be the corresponding reflection operator. Let $N \in \mathbb{N}_{>0}$, $\beta \geq 0$, and let $\langle \cdot\rangle_{L,N,\beta}$ be the expectation operator associated to the spin $O(N)$ (recall \eqref{eq:deffunctional}).
For any pair of functions
$f, g \in \mathcal{B}^+$, we have that,
\begin{enumerate}
\item[(1)] $\langle \,  f \, \Theta g \,  \rangle_{L, N, \beta} = \langle \,  g\,  \Theta f \, \rangle_{L, N, \beta}$
\item[(2)] $ \langle \, f \,  \Theta f\,  \rangle_{L, N, \beta} \geq 0$.
\end{enumerate}
From this we obtain that,
\begin{equation}\label{eq:RPCSsites}
\langle f \, \Theta g \rangle_{L, N, \beta}
\leq 
\langle f \, \Theta f \rangle_{L, N, \beta}^{\frac{1}{2}}
\, \, 
\langle g \, \Theta g \rangle_{L, N, \beta}^{\frac{1}{2}}.
\end{equation}
\end{prop}

Recall from Proposition \ref{prop:equivalence} that when $U$ is chosen according to \eqref{eq:weightSpin} we have for $x\neq y$
$$
\langle\varphi_x^1\varphi_y^1\rangle_{L,N,\beta}=G_{L,N,\beta,U}(x,y).
$$
Our current aim is to prove complementary results to Propositions 
\ref{prop:general sufficient} and
\ref{prop:superharmonicity} in the context of the $O(N)$ spin model with reflections through sites. To begin we prove a complementary result to Theorem \ref{thm:RPaverage}. Recall that for $A\subset \T_L$ and a reflection $\Theta$ we define $A^\pm=(A\cap\T_L^\pm)\cup(\Theta(A\cap\T_L^\pm))$ (we use this same definition for reflections through sites or edges).

\begin{prop}\label{prop:RPaveragespins}
Under the same assumptions of Proposition \ref{prop:RPsites} we have
\begin{equation}
\sum_{\substack{x,y\in A\\ x\neq y}}\langle \varphi^1_x \varphi^1_y\rangle_{L,N,\beta}\leq \frac{1}{2}\left( \sum_{\substack{x,y\in A^+\\ x\neq y}}\langle \varphi^1_x \varphi^1_y\rangle_{L,N,\beta}+\sum_{\substack{x,y\in A^-\\ x\neq y}}\langle \varphi^1_x \varphi^1_y\rangle_{L,N,\beta}\right).
\end{equation}
\end{prop}
\begin{proof}
First we apply Proposition \ref{prop:RPsites} for a plane $R$ through sites with associated reflection operator $\Theta$. For $\eta>0$ we take 
\begin{align}
f(\varphi)&=\prod_{x\in A\cap (\T_L^+\setminus \T_L^R)}\big(1+\eta\varphi^1_x\big)\prod_{x\in A\cap\T_L^R}\big(1+\eta\varphi^1_x\big)^{\tfrac12},
\\
g(\varphi)&=\prod_{x\in A\cap(\T_L^-\setminus\T_L^R)}\big(1+\eta\varphi^1_{\Theta x}\big)\prod_{x\in A\cap\T_L^R}\big(1+\eta\varphi^1_{\Theta x}\big)^{\tfrac12}.
\end{align}
If $\eta\leq1$ then these functions are non-negative and there is no issue with taking the square root of $1+\eta\varphi^1_x$.
Note that $f,g\in\Bcal^+$, hence we may use Proposition \ref{prop:RPsites}.
We have $\langle f \, \Theta g \rangle_{L, N, \beta}=\langle \prod_{x\in A}\big(1+\eta\varphi^1_x\big)\rangle_{L, N, \beta}$, $\langle f \, \Theta f \rangle_{L, N, \beta}=\langle\prod_{x\in A^+}\big(1+\eta\varphi^1_x\big)\rangle_{L, N, \beta}$ and $\langle g \, \Theta g \rangle_{L, N, \beta}=\langle\prod_{x\in A^-}\big(1+\eta\varphi^1_x\big)\rangle_{L, N, \beta}$ (here we used that $A^-$ is symmetric with respect to $R$, hence we may replace $\varphi^1_{\Theta x}$ with $\varphi^1_x$ in $\langle g \, \Theta g \rangle_{L, N, \beta}$). From this we obtain 
$$
\langle f \, \Theta g \rangle_{L, N, \beta}  =1+\eta\sum_{x\in A}\langle \varphi^1_x\rangle_{L, N, \beta}+\eta^2\sum_{\substack{x,y\in A\\ x\neq y}}\langle \varphi^1_x\varphi^1_y \rangle_{L, N, \beta}+O(\eta^3)=1+\eta^2\sum_{\substack{x,y\in A\\ x\neq y}}\langle \varphi^1_x\varphi^1_y \rangle_{L, N, \beta}+O(\eta^3),
$$
where we used that $\langle \varphi^1_z\rangle_{L, N, \beta}=0$ for every $z\in\T_L$. We also have the corresponding equalities for $\langle f \, \Theta f \rangle_{L, N, \beta}$ and $\langle g \, \Theta g \rangle_{L, N, \beta}$.
Now we use Proposition \ref{prop:RPsites} and the expansion $(1+x)^{\tfrac12}=1+\tfrac{x}{2}+O(x^2)$ in the same way as in the proof of Theorem \ref{thm:RPaverage} to obtain that
\begin{equation}
\begin{aligned}
\langle f \, \Theta g \rangle_{L, N, \beta}&=1+\eta^2\sum_{\substack{x,y\in A\\ x\neq y}}\langle \varphi^1_x\varphi^1_y \rangle_{L, N, \beta}+O(\eta^3)
\\
&\leq\bigg(\Big(1+\eta^2\sum_{\substack{x,y\in A^+\\ x\neq y}}\langle \varphi^1_x\varphi^1_y \rangle_{L, N, \beta}+O(\eta^3)\Big)\Big(1+\eta^2\sum_{\substack{x,y\in A^-\\ x\neq y}}\langle \varphi^1_x\varphi^1_y \rangle_{L, N, \beta}+O(\eta^3)\Big)\bigg)^{\tfrac12}
\\
&=\bigg(1+\eta^2\Big(\sum_{\substack{x,y\in A^+\\ x\neq y}}\langle \varphi^1_x\varphi^1_y \rangle_{L, N, \beta}+\sum_{\substack{x,y\in A^-\\ x\neq y}}\langle \varphi^1_x\varphi^1_y \rangle_{L, N, \beta}\Big)
+O(\eta^3)\bigg)^{\tfrac12}
\\
&=1+\frac{\eta^2}{2}\Big(\sum_{\substack{x,y\in A^+\\ x\neq y}}\langle \varphi^1_x\varphi^1_y \rangle_{L, N, \beta}+\sum_{\substack{x,y\in A^-\\ x\neq y}}\langle \varphi^1_x\varphi^1_y \rangle_{L, N, \beta}\Big)+O(\eta^4).
\end{aligned}
\end{equation}
Now by inspecting the $\eta^2$ term we see, by taking $\eta$ sufficiently small, that the result follows.
\end{proof}
It is worth noting that in the case when $A=\{x,y\}$ and the reflection plane $R$ is such that $x\in\T_L^+$ and $y\in\T_L^-$ Proposition \ref{prop:RPaveragespins} becomes
\begin{equation}\label{eq:averagespin2sites}
\langle\varphi_x^1\varphi_y^1\rangle_{L,N,\beta}\leq \frac{\langle\varphi_x^1\varphi_{\Theta x}^1\rangle_{L,N,\beta}+\langle\varphi_{\Theta y}^1\varphi_y^1\rangle_{L,N,\beta}}{2},
\end{equation}
and is analogous to (\ref{eq:stat1corrineq}).

Now we present our complementary results to Propositions 
\ref{prop:general sufficient} and
\ref{prop:superharmonicity}.
 The only difference in the statements is that `odd' is replaced by `even'
and that the proposition holds only for the weight function $U$ as in Proposition \ref{prop:equivalence}. Under this choice, the random path model is a representation of the spin O(N) model.
\begin{figure}
\includegraphics[scale=1.2]{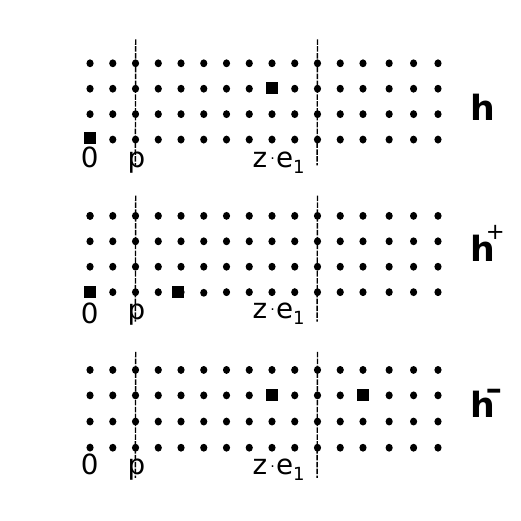}
\includegraphics[scale=1.2]{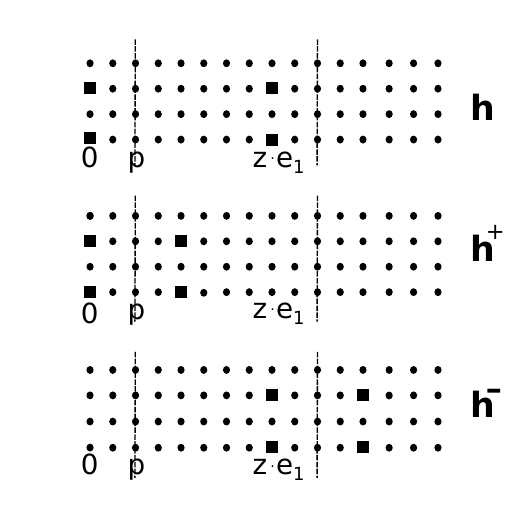}
\centering
\caption{A representation of three slices of a torus of side length $L=16$, the dashed line pass through the sites of $\T_L^R$. The vector fields equal 1 on the square vertices and 0 on the round vertices.
}
\label{Fig:reflectionsthroughsites}
\end{figure}
\begin{prop}\label{prop:general sufficientsites}
Let $L \in 2 \mathbb{N}$,  let $z \in \T_L$ be an arbitrary point such that $z \neq o$, let $\boldsymbol{e}_i$ be a cartesian vector,
let $\beta \geq 0$, $N \in \mathbb{N}_{>0}$, and
let $U$ be given by \eqref{eq:weightSpin}. 
The following inequalities holds for any integer $q \in \mathbb{N}$ such that  $z\cdot\be_i+q$ is even and such that $z\cdot\be_i-q, z\cdot\be_i+q \in (0, L)$ 
\begin{align}
\label{eq:super-harmonicitytwosites}
& G_{L,N,\beta,U}( z)  \leq \frac{1}{2} G_{L,N,\beta,U}\big( (z\cdot\be_i-q) \, \boldsymbol{e_i}  \big )   + \frac{1}{2} G_{L,N,\beta,U}\big (  (z\cdot\be_i+q ) \, \boldsymbol{e_i}  \big ), \\
\label{eq:super-harmonicitysites}
& 
G_{L,N,\beta,U}^{\boldsymbol{e}_i} \big ( z  \big )  - 
G_{L,N,\beta,U}^{\boldsymbol{e}_i}\big (  z - q \, \boldsymbol{e}_i \big ) \leq 
G_{L,N,\beta,U}^{\boldsymbol{e}_i} \big (  z + q \,  \boldsymbol{e_i} \big )  - 
G_{L,N,\beta,U}^{\boldsymbol{e}_i} \big (  z \big ).
\end{align}
\end{prop}
\begin{proof}
The inequality \eqref{eq:super-harmonicitytwosites} follows from Proposition \ref{prop:RPaveragespins} applied with $x=o$, $y=z$ and taking the reflection in the plane $R=\{x\in \R\, : \, x\cdot\be_i=\tfrac12(z\cdot\be_i+q)\}$ which requires that $z\cdot\be_i+q$ is even. We have $\langle\varphi_o^1\varphi_{\Theta o}^1\rangle_{L,N,\beta}=\langle\varphi_o^1\varphi_{(z\cdot\be_i+q)\be_i}^1\rangle_{L,N,\beta}$ and $\langle\varphi_z^1\varphi_{\Theta z}^1\rangle_{L,N,\beta}=\langle\varphi_z^1\varphi_{z+(z\cdot\be_i+q-2z\cdot\be_i)\be_i}^1\rangle_{L,N,\beta}=\langle\varphi_o^1\varphi_{(q-z\cdot\be_i)\be_i}^1\rangle_{L,N,\beta}=\langle\varphi_o^1\varphi_{(z\cdot\be_i-q)\be_i}^1\rangle_{L,N,\beta}$ where we used symmetries of the torus. After applying Proposition \ref{prop:RPaveragespins} in the form given by \eqref{eq:averagespin2sites} (which requires that $z\cdot\be_i-q>0$ so that $o$ and $z$ are in different halves of the torus) and then using Proposition \ref{prop:equivalence} we obtain the result.
For \eqref{eq:super-harmonicitysites} the result will follow from the inequality,
\begin{equation}
\begin{aligned}
\label{eq:firstrelationsites}
& \langle \varphi_o \cdot \varphi_{(z\cdot\be_i)\be_i} \rangle_{L, N, \beta}+\langle \varphi_o \cdot \varphi_ z \rangle_{L, N, \beta}
\\
 &\leq
\frac{1}{2}\bigg(\langle \varphi_o \cdot \varphi_ { (z\cdot\be_i+q ) \, \boldsymbol{e_i} } \rangle_{L, N, \beta}+\langle \varphi_o \cdot \varphi_ { (z\cdot\be_i-q) \boldsymbol{e_i}} \rangle_{L, N, \beta}+
\langle \varphi_o \cdot \varphi_ {z+q\be_i} \rangle_{L, N, \beta}+ 
\langle \varphi_o \cdot \varphi_ {z-q\be_i} \rangle_{L, N, \beta}
\bigg),
\end{aligned}
\end{equation}
after rearranging (just as in the proof of Proposition \ref{prop:superharmonicity}) and then using Proposition \ref{prop:equivalence} to move to the path model. We take the same plane $R=\{x\in \R\, : \, x\cdot\be_i=\tfrac12(z\cdot\be_i+q)\}$ with its associated reflection operator $\Theta$ as previously. Consider the set $$
A=\{o,z,(z\cdot\be_i)\be_i,z-(z\cdot\be_i)\be_i\}.
$$
 If we define $A^\pm=(A\cap\T_L^\pm)\cup(\Theta(A\cap\T_L^\pm))$ we have 
 $$
 A^+=\{o,(z\cdot\be_i+q)\be_i,z-(z\cdot\be_i)\be_i,z+q\be_i\},
 $$
 $$
 A^-=\{z,(z\cdot\be_i)\be_i,q\be_i,z+(q-z\cdot\be_i)\be_i\}.
 $$
Now recall from Proposition \ref{prop:equivalence} that $\langle\varphi_x^1\varphi_y^1\rangle_{L,N,\beta}=G_{L,N,\beta,U}(x,y)$ for an appropriate choice of $U$.  Applying Proposition \ref{prop:RPaveragespins} and using translation invariance we obtain \eqref{eq:firstrelationsites}. Now we use Proposition \ref{prop:equivalence} to move back to the random path model, giving the result.
\end{proof}

We now present the proof of Theorem \ref{theo:monotonicityspinO(N)}. 
Contrary to Theorem \ref{theo:monotonicity}, the statement is proved only for the spin $O(N)$ model. The reflection positivity of the model for reflections through sites allows the derivation of a full monotonicity property, which is not just limited to odd sites.

\begin{proof}[\textbf{Proof of Theorem \ref{theo:monotonicityspinO(N)}}]
The first inequality follows from Theorem \ref{theo:monotonicity} at odd sites and from a direct application of
(\ref{eq:averagespin2sites}) at even sites.
The proof of the second inequality is analogous to the proof of the second inequality in Theorem \ref{theo:monotonicity}.
To begin,  write $G^{\be_i}_L=G^{\be_i}_{L,N,\beta,U}$ where $U$ is given by \eqref{eq:weightSpin}. From Propositions \ref{prop:superharmonicity} and \ref{prop:general sufficientsites} with $q=1$ we have that, for any $y \in \T_L$
such that $y\cdot\be_i\pm 1\in(0,L)$
\begin{equation}\label{eq:deduction}
G^{\be_i}_L(y+ \be_i)-G^{\be_i}_L(y)\geq G^{\be_i}_L(y)-G^{\be_i}_L(y- \be_i),
\end{equation}
Note that here we have no restriction on the parity of $y\cdot\be_i$.
Suppose that for some $z$ such that $z\cdot\be_i=0$ and some $n\in(0,L/2],$
$$
G^{\be_i}_L(z+(n+1)\be_i)>G^{\be_i}_L(z+n\be_i).
$$
This will lead us to a contradiction. Indeed, 
from (\ref{eq:deduction}) we deduce that, $G^{\be_i}_L(z+(n+1)\be_i)<G^{\be_i}_L(z+(n+2)\be_i)<G^{\be_i}_L(z+(n+3)\be_i)<\dots < G^{\be_i}_L(z-(n+1)\be_i)=G^{\be_i}_L(z+(n+1)\be_i)$, where in the last steps we used the symmetry of the torus.
This contradiction completes the proof of the second inequality.
 \end{proof}
 
 \begin{rem}
 For the proof of the previous theorem we used the inequalities derived from Proposition \ref{prop:reflection positivity 2}, which used reflection through edges in the context of interacting paths,
 and those derived from Proposition \ref{prop:RPaveragespins}, which used reflection through sites in the context of spins.
 For the spin O(N) model it would be  possible to derive Proposition  \ref{prop:reflection positivity 2} without using its representation as a system of interacting paths, using reflection through edges in a classical way (see for example \cite[Chapter 10]{FriedliVelenik}). Thus, representing the spin O(N) model as a system of interacting paths is not really necessary for the derivation of Theorem
 \ref{theo:monotonicityspinO(N)}.
 \end{rem}

\subsection{Proof of Theorem \ref{theo:pointwise}}
The main goal of this section is to present the proof of Theorem \ref{theo:pointwise}, which is presented at the end of the section.
For any $z \in \mathbb{Z}^d$, we define the box with $z$ as corner, 
$ 
\mathbb{Q}_z := \big \{ (x_1, \ldots,  x_d) \, \, \in \mathbb{Z}^d : \, \,  \forall i \in \{1, \ldots, d\}, \, \, 
x_i \leq |z_i| $ or $x_i > L - |z_i|   \big  \}.
$ 
It will be necessary to consider vertices that lie on certain $d-1$ dimensional hyperplanes of $\T_L$, but not on a cartesian axis, separately from other vertices. Due to this necessity we define
\begin{equation}
\Hcal_L:=\{x=(x_1,\dots,x_d)\in \T_L\, \, : \, \, x_1,x_2,\dots,x_d\neq 0\text{ or } \exists i\text{ s.t. }x_i\neq 0,\, x_j=0\, \forall j\neq i\},
\end{equation}
to be the set of vertices with all non-zero coordinates together with the cartesian axes.
For notational reasons, in the sequel we will omit the sub-script from $\langle \varphi^1_o \varphi^1_z \rangle_{L, N, \beta}$ when appropriate. The next proposition is a consequence of Theorem \ref{theo:monotonicityspinO(N)} and applies to vertices in $\Hcal_L$.
\begin{prop}\label{prop:goinginside}
Under the same assumptions of Proposition \ref{prop:RPsites}, we have that,
for any $z \in \T_L\cap \Hcal_L$ such that $\|   z  \|_{ \infty} \leq \frac{L}{2}$,
\begin{equation}\label{eq:conclusioninside}
\langle \varphi^1_o \varphi^1_z \rangle_{L, N, \beta}  \geq 1/N - \delta \implies  \langle \varphi^1_o \varphi^1_y  \rangle_{L, N, \beta} \geq 1/N - 2^d\,  \delta, 
\quad  \forall y \in \mathbb{Q}_z\cap \Hcal_L.
\end{equation}
\end{prop}
\begin{proof}
Assume that $L \in 2 \mathbb{N}$,  $z \in \T_L$,
$\langle \varphi^1_o \varphi^1_z  \rangle \geq 1/N - \delta$,
and that $y \in \mathbb{Q}_z\cap \Hcal_L$.
We will prove the statement under the assumption that
 $z, y \in \T_L$ are such that  $z \cdot \boldsymbol{e}_i > 0$, and
$y \cdot \boldsymbol{e}_i > 0$ for every $i \in \{1, \ldots, d\}$.
By the torus symmetry, this will imply (\ref{eq:conclusioninside}) for any $z \in \T_L\cap \Hcal_L$ and $y \in \mathbb{Q}_z\cap \Hcal_L.$
Also we will assume that $y \neq o$, in which case the proposition trivially holds since $\langle \varphi^1_o \varphi^1_y  \rangle_{L, N, \beta}  = 1$. If $y$ lies on a coordinate axis then the result is automatic by Theorem \ref{theo:monotonicityspinO(N)}. Suppose $y$ does not lie on a coordinate axis,
 for each $i \in \{1, \ldots, d\}$, define
$$
D_i := z_i - y_i,
$$
and note that, by assumption, $D_i \in \mathbb{N}$.
Since $z,y \in \mathbb{Z}^d$, there must exist a path of nearest neighbour sites of $\mathbb{Z}^d$ consisting of at most $d$ segments,
$$
 (z^1_0, z^1_1, \ldots, z^1_{D_1}), \quad 
 (z^2_0, z^2_1, \ldots, z^2_{D_2}),  \quad 
 \ldots  \quad 
 (z^d_0, z^d_1, \ldots, z^d_{D_d}),
$$
 such that, for each $i \in \{1, \ldots, d\}$ and $j \in [1, D_i]$,
 $$
z^i_{j-1} -  z^i_{j} =  \boldsymbol{e}_i, \quad z^i_j \in \mathbb{Z}^d,
 $$
 and, for any $i \in \{1, \ldots, d-1\}$,
 \begin{equation}\label{eq:claim11}
  z_{D_i}^{i}  = z_{0}^{i+1}, \quad  \quad z_0^1 = z, \quad  \quad z_{D_d}^{d} = y.
 \end{equation}
See for example Figure \ref{Fig:pathandslices}.
\begin{figure}[t]
\includegraphics[scale=0.41]{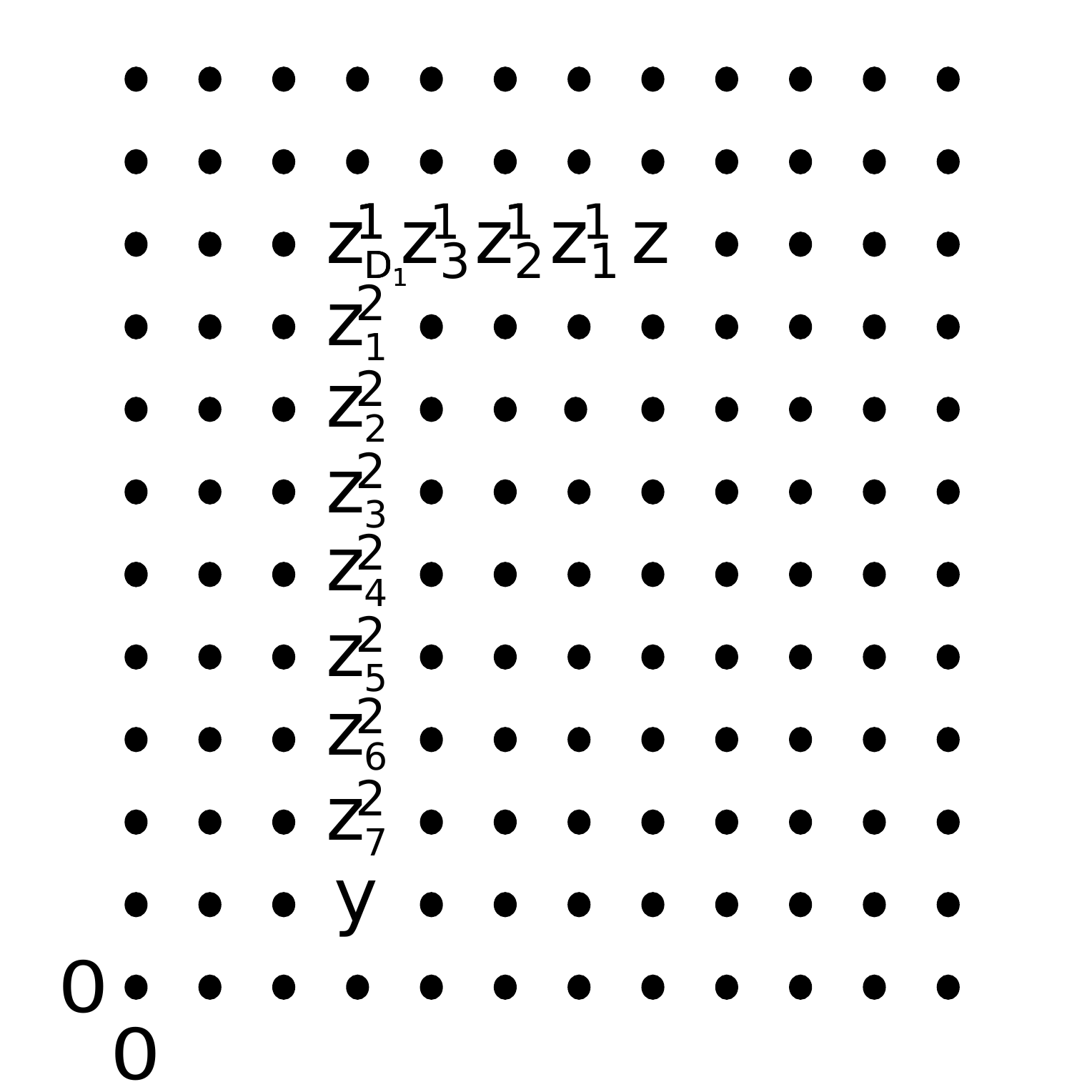}
\includegraphics[scale=0.41]{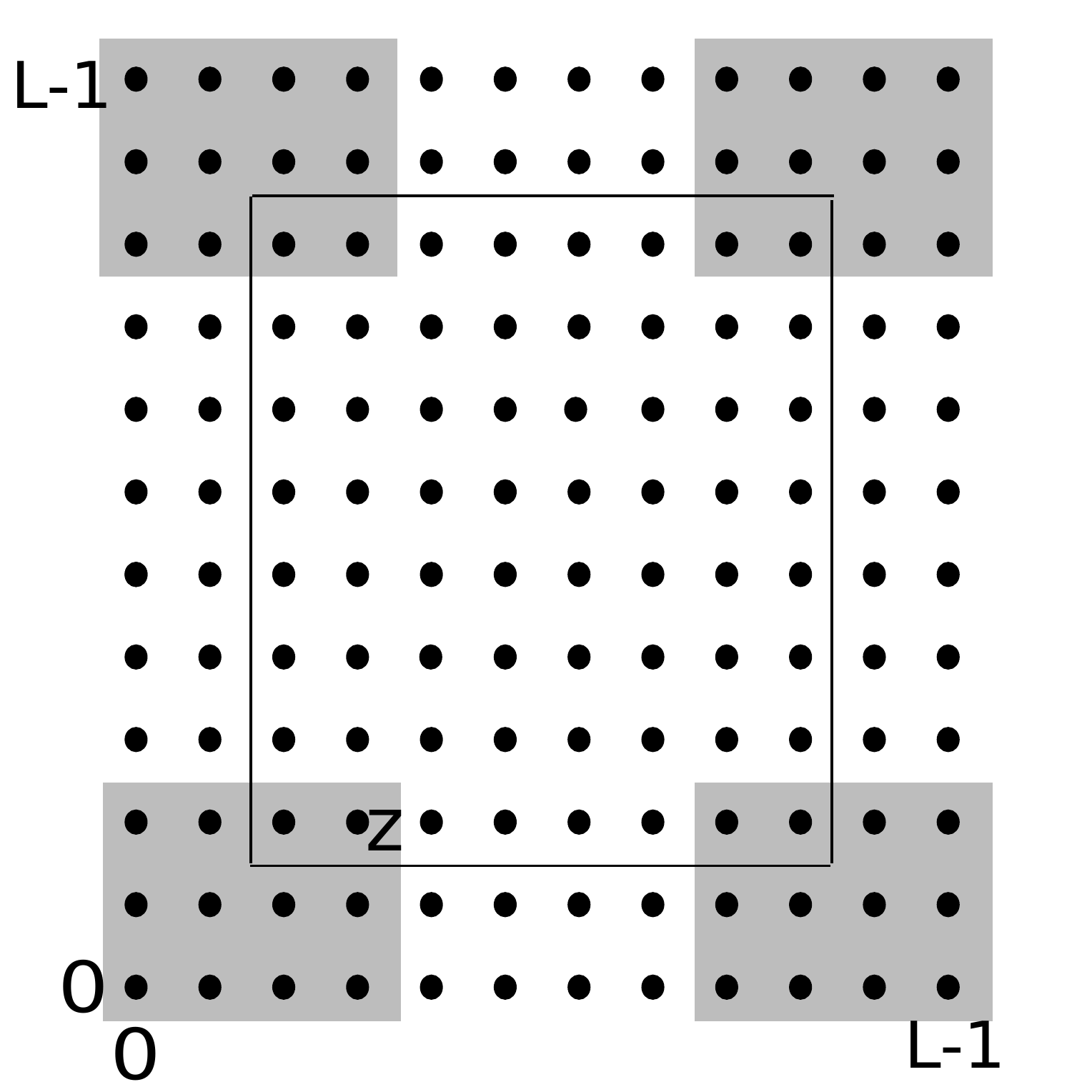}
\centering
\caption{
\textit{Left:} A slice of the torus
and a path in $\mathbb{Z}^d$ consisting of two segments and connecting $z \in \T_L$ to $y \in \mathbb{Q}_z$ are represented. We assume that the side length of the torus is much larger than 12.
\textit{Right:} 
A torus of side length $L =12$. The straight lines are the boundaries of the set $\mathbb{S}_{2, L}$ (periodic boundary conditions are taken into account)
and the dark region corresponds to the set $\mathbb{Q}_z$.
}
\label{Fig:pathandslices}
\end{figure}
We claim that, for any $i \in \{1, \ldots, d\}$,
\begin{equation}\label{eq:claim0}
\langle \varphi^1_o \varphi^1_{z^{i}_{D_{i}}} \rangle  \, \geq \, 
2 \, \langle \varphi^1_o \varphi^1_{z^{i}_{0}} \rangle \, - \, 1/N.
\end{equation}
The claim implies the proposition, since,
using (\ref{eq:claim0})  $d$ times and recalling (\ref{eq:claim11}), we obtain that,
\begin{align*}
\langle \varphi^1_o \varphi^1_y \rangle &  =
\langle \varphi^1_o \varphi^1_{z^{d}_{D_{d}}} \rangle 
\geq 
2  \langle \varphi^1_o \varphi^1_{z^{d}_{0}} \rangle  - 1/N 
= 2  \langle \varphi^1_o \varphi^1_{z^{d-1}_{D_{d-1}}} \rangle  - 1/N   \\
 & \geq \ldots \geq 
 2^d\, \langle \varphi^1_o \varphi^1_{z^{1}_{0}} \rangle  \, - \,  (2^d -1)/N 
 =
2^d\, \langle \varphi^1_o \varphi^1_{z} \rangle  \, - \,  (2^d -1)/N 
\geq 1/N -  2^d \, \delta.
\end{align*}
We now prove  (\ref{eq:claim0}).
For the next inequalities we use both inequalities in Theorem \ref{theo:monotonicityspinO(N)} which applies as $y\in \Hcal_L$ so we do not use the ``n=0" case,
$$
2 \langle \varphi^1_o \varphi^1_{z^i_{0}} \rangle
\leq 
\langle \varphi^1_o \varphi^1_{z^i_{0}} \rangle +
\langle \varphi^1_o \varphi^1_{(z^i_{0} \cdot \boldsymbol{e}_i) \boldsymbol{e}_i)} \rangle,
\quad 
\langle \varphi^1_o \varphi^1_{z^i_{0}} \rangle
+ \langle \varphi^1_o \varphi^1_{(z^i_{0} \cdot \be_i) \be_i} \rangle
\leq 
\langle \varphi^1_o \varphi^1_{z^i_{D_i}} \rangle
+ \langle \varphi^1_o \varphi^1_{(z^i_{D_i} \cdot \be_i) \be_i} \rangle.
$$
Since by symmetry we have that 
$
\forall z \in \T_L$,  $\langle \varphi^1_o \varphi^1_{z} \rangle =
\frac{\langle \varphi_o \cdot \varphi_{z} \rangle}{N} \leq 1/N,
$ combining the two inequalities above we deduce  (\ref{eq:claim0}) and conclude the proof.
\end{proof}

Now we turn to vertices in $\T_L\setminus \Hcal_L$.

\begin{prop}\label{prop:goinginsideH}
Under the same assumptions of Proposition \ref{prop:RPsites}, we have that,
for any $z \in \T_L \cap \Hcal_L$ such that $\|   z  \|_{ \infty} \leq \frac{L}{2}$,
\begin{equation}\label{eq:conclusioninside2}
\langle \varphi^1_o \varphi^1_z \rangle_{L, N, \beta}  \geq 1/N - \delta \implies  \langle \varphi^1_o \varphi^1_x  \rangle_{L, N, \beta} \geq (1/N - 2^d\,  \delta)\left(\frac{\beta e^{-2d\beta}}{N}\right)^{d-2}, 
\quad  \forall x \in \mathbb{Q}_z \setminus \Hcal_L.
\end{equation}
\end{prop}
\begin{proof}
To begin, recall the definition of $\mathcal{S}_A$, which was introduced in Definition \ref{def:partition functions}.
For $x\in \mathbb{Q}_z\setminus\Hcal_L$ consider $F:\Scal_{\{o,x\}}\to \cup_{y\sim x} \Scal_{\{o,y\}}$ that acts on $w\in\Scal_{\{o,x\}}$ by removing the last link of the 1-walk from $o$ to $x$ (the link incident to $x$). Note that $F$ is well defined as the 1-walk must have at least two links due to $\mathbb{Q}_z \setminus \Hcal_L$ not containing neighbours of $o$. We claim that $F$ is a bijection. Indeed, $F$ is injective as if two configurations differ at the last links of their 1-walks then their images under $F$ are in different $\Scal_{\{o,y\}}$'s. On the other hand, if the configurations differ elsewhere then their images under $F$ still differ as these links are not changed by $F$. Also $F$ is surjective as any $w\in \Scal_{\{o,y\}}$, $y\sim x$, is the image of the configuration that coincides with $w$ except that the 1-walk is extended by one link on $\{x,y\}$.

Now consider the subset of configurations in $\Scal_{\{o,x\}}$ such that there is only one link incident to $x$ (the last link of the 1-walk). For such a configuration $\mu_{\T_L,N,\beta,U}(w)$ differs from $\mu_{\T_L,N,\beta,U}(F(w))$ by a factor of $(\beta/2)(2/N)=\beta/N$ and $F(w)$ has no links incident to $x$. This gives
\begin{equation}
\mu_{\T_L,N,\beta,U}(\Scal_{\{o,x\}})\geq \frac{\beta}{N}\sum_{y\sim x} \mu_{\T_L\setminus \{x\},N,\beta,U}(\Scal_{\{o,y\}}),
\end{equation}
where we used the notation  $\T_L \setminus A$ for the graph which is obtained from $(\T_L, \E_L)$ by removing all the sites $A \subset \T_L$ and all the edges which are incident to it.
Now we use the spin representation. For any $\varphi\in (\mathbb{S}^{N-1})^{\T_L}$ we have $-2d\leq \sum_{y\sim x}\varphi_x\cdot\varphi_y\leq 2d$, hence
\begin{equation}
\begin{aligned}
Z^{spin}_{\T_L,N,\beta}\langle\varphi^1_o\varphi^1_x\rangle_{\T_L,N,\beta} \geq & \frac{\beta}{N}\sum_{y\sim x}  Z^{spin}_{\T_L\setminus \{x\},N,\beta}\langle\varphi^1_o\varphi^1_y\rangle_{\T_L\setminus \{x\},N,\beta}
\\
\geq &\frac{\beta}{N}\sum_{y\sim x} e^{-2d\beta}Z^{spin}_{\T_L,N,\beta}\langle\varphi^1_o\varphi^1_y\rangle_{\T_L,N,\beta}
\\
\geq &\frac{\beta e^{-2d\beta}}{N}\sum_{y\sim x} Z^{spin}_{\T_L,N,\beta}\langle\varphi^1_o\varphi^1_y\rangle_{\T_L,N,\beta},
\end{aligned}
\end{equation}
From which we obtain
\begin{equation}
\langle\varphi^1_o\varphi^1_x\rangle_{\T_L,N,\beta}\geq \frac{\beta e^{-2d\beta}}{N}\sum_{y\sim x} \langle\varphi^1_o\varphi^1_y\rangle_{\T_L,N,\beta}.
\end{equation}
Now if there is a $y\sim x$ such that $y\in \mathbb{Q}_z\cap \Hcal_L$ then we are done by Proposition \ref{prop:goinginside}, however this may not be the case. However, it is easily seen that for any $x\in\mathbb{Q}_z\setminus \Hcal_L$ there is a $y\in  \mathbb{Q}_z\cap \Hcal_L$ such that $\|x-y\|_1\leq d-2$. Hence we can repeat the same bound for $y\sim x$, (and possibly $y_1\sim y$ etc) until we have a site with at least one neighbour in $\mathbb{Q}_z\setminus \Hcal_L$ and then apply Proposition \ref{prop:goinginside} to this neighbour to obtain the result.
\end{proof}

The next lemma states that, if the Ces\`aro mean of the two-point function is close enough to $1/N$, which is a uniform upper bound of the two-point function, then there exists a vertex $z$ which is `far enough away' from any cartesian axis such that 
$\langle \varphi^1_o \varphi^1_{z} \rangle$ is `reasonably close' to $1/N$ as well. 
Define
$
\mathbb{B}_r := 
\big \{ z \in \mathbb{Z}^d : \| z \|_{\infty} \leq r   \big  \}
$,
where $\| \, \cdot \, \|_{\infty}$ is with respect to the  torus metric.
\begin{lem}\label{lem:positiveaverageimplies}
Suppose that $d \geq 2$ and $L\in2\N$. Assume that there exists a constant $\delta \in (0, 1)$ such that, 
\begin{equation}\label{eq:Condition2}
 \frac{1}{|\T_{L}|} \sum\limits_{z \in \T_{L}} \langle \varphi^1_o \varphi^1_{z} \rangle \geq 1/N - \delta.
\end{equation}
Then,
\begin{equation}\label{eq:Conclusion}
\begin{aligned}
&\forall y \in \mathbb{B}_{L/8}\cap \Hcal_L  \quad \langle \varphi^1_o \varphi^1_{y} \rangle \geq 1/N  - 2^{2d} \, \delta.
\\
&\forall y \in \mathbb{B}_{L/8} \setminus \Hcal_L  \quad \langle \varphi^1_o \varphi^1_{y} \rangle \geq (1/N  - 2^{2d} \, \delta)\left(\frac{\beta e^{-2d\beta}}{N}\right)^{d-2}.
\end{aligned}
\end{equation}
\end{lem}
\begin{proof}
In the whole proof we will assume that $d \geq 2$.
For any $r, L  \in \mathbb{N}$, we define the set
$$
 \mathbb{S}_{r,L} := 
 \big \{     
 z \in \T_L \, \, : \, \,\exists i \in \{1, \ldots, d\} \, \, \mbox{s.t.}  \, \, z \cdot \boldsymbol{e}_i < r \mbox{ or } 
 L - z \cdot \boldsymbol{e}_i \geq r
 \big \}.
 $$
 See Figure \ref{Fig:pathandslices} for a graphical representation of $\mathbb{S}_{r,L}$.
A simple computation shows that, if $L \in 2 \mathbb{N}$, and $r \in (0, L/2) \cap \mathbb{N}$, then,
\begin{equation}\label{eq:cardinalityslice}
 | \T_L \setminus  \mathbb{S}_{r,L} | =   (L-2r)^d.
 \end{equation}
From now on we set $r=L/8$ (which may not be an integer) and $L\in 2\N$.
We claim that, under the assumptions of the theorem, the following holds,
\begin{equation}\label{claim1}
\exists z_L \in 
\T_{L} \setminus 
\mathbb{S}_{r,L} \, \, \, \mbox{ s.t. } \, \,  \langle \varphi^1_o  \varphi^1_{z_L} \rangle \geq 1/N- 2^d \,  \delta.
\end{equation}
The proof of claim (\ref{claim1}) is by contradiction. Assume that (\ref{claim1}) is false, namely that 
$$
 \forall z \in \T_L \setminus \mathbb{S}_{r, L} \quad  
 \langle \varphi^1_o  \varphi^1_{z} \rangle  < 1/N - 2^d  \, \delta,
$$
under the assumptions of the theorem.
Then, (\ref{eq:cardinalityslice}) (and recalling that we have set $r=L/8$), we obtain that,
\begin{align*}
\sum\limits_{z \in \T_L} \langle \varphi^1_o  \varphi^1_{z} \rangle   & < \, \, \big | \T_L \setminus \mathbb{S}_{r, L} \big | \, (\tfrac1N - 2^d \delta) \,\,  + \, \,
\tfrac1N \big ( \, |\T_L| - \big | \T_L \setminus \mathbb{S}_{r, L}| \big )
 \\
& = 
L^d \, \,  \Big (\, \frac{1}{N} \, -\,  2^d \, \delta \, \big(\frac{3}{4}\big)^d \, \Big ) <
L^d \,  (\, \frac{1}{N}  -   \delta  )
 \end{align*}
  This violates the hypothesis of the theorem and, thus, we obtain the desired contradiction. This proves (\ref{claim1}).
Note that since $z_L \in  \T_L \setminus \mathbb{S}_{r,L}$, we have that $z\in \Hcal_L$ and $\mathbb{Q}_{z_L} \supset \mathbb{B}_{r}.$
From  (\ref{claim1}) and Propositions \ref{prop:goinginside} and \ref{prop:goinginsideH} we deduce that, 
\begin{equation*}
\begin{aligned}
&\forall y \in  \mathbb{B}_{L/8}\cap \Hcal_L, \quad  \langle \varphi^1_o  \varphi^1_{y} \rangle \geq  1/N - 2^{2d}  \delta,
\\
&\forall y \in \mathbb{B}_{L/8} \setminus \Hcal_L  \quad \langle \varphi^1_o \varphi^1_{y} \rangle \geq (1/N  - 2^{2d} \, \delta)\left(\frac{\beta e^{-2d\beta}}{N}\right)^{d-2}.
\end{aligned}
\end{equation*}
This concludes the proof.
\end{proof}

The next theorem is a very classical result which was proved in  \cite{FrohlichSimonSpencer}. 
\begin{thm}[Fr\"ohlich, Simon and Spencer (1976)]\label{theo:classical}
Consider the spin $O(N)$ model on the torus of side length $L$ identified with $\Z^d/L\Z^d$, with inverse temperature $\beta \geq 0$, and  $N \in \mathbb{N}_{>0}$.
When $d \geq 3$,  there exists $\beta_0 < \infty$ such that, for any $\beta\geq\beta_0$, 
\begin{equation}
\liminf_{\substack{L\to\infty : \\ L \mbox{ \footnotesize even }}}\frac{1}{|\T_{L}|}  \sum\limits_{z \in \T_{L}} \langle  \varphi^1_o \, \,  \varphi^1_z \rangle_{L, N, \beta}\geq \frac{1}{N} - \frac{\beta_0}{N} \frac{1}{\beta}.
\end{equation}
\end{thm}

We are now ready to prove Theorem \ref{theo:pointwise}.
\begin{proof}[\textbf{Proof of Theorem \ref{theo:pointwise}}]
Let $\beta_0$ be the same constant as in Theorem \ref{theo:classical}, let  $\delta >0 $ be such that  $2^{2d} \delta < 1/N$. 
If $\beta$ is large enough such that $\beta_0/(N \beta) < \delta$, we deduce from  Theorem
 \ref{theo:classical}  that  there exists  $L_0  = L_0(\beta, d) < \infty$ such that for any even $L > L_0$, 
$$
\frac{1}{|\T_L|} \sum\limits_{z \in \T_L} \langle\varphi^{1}_o\varphi^{1}_z\rangle_{L,N,\beta}   \geq \frac{1}{N} -   \delta.
$$ 
Since 
$\langle \varphi^1_o \varphi^1_z \rangle \leq 1/N$ for any $z \in \T_L$,
applying Lemma \ref{lem:positiveaverageimplies} we deduce that for any even $L > L_0$, any $z \in \mathbb{B}_{L/8}\cap \Hcal_L$, 
$\langle \varphi^1_o \varphi^1_z \rangle \geq 1/N - 2^{2d} \delta > 0$ and for any $z \in \mathbb{B}_{L/8}\setminus \Hcal_L$, 
$\langle \varphi^1_o \varphi^1_z \rangle \geq (1/N - 2^{2d} \delta)(\beta e^{-2d\beta}/N)^{d-2} > 0$.
This concludes the proof.
\end{proof}

\section*{Acknowledgements}
B. Lees acknowledges support from the Alexander von Humboldt foundation.
 L. Taggi acknowledges support  from the DFG German Research Foundation BE 5267/1 and from the  EPSRC Early Career Fellowship EP/N004566/1.
The authors thank the two anonymous referees for carefully reading the paper and their useful suggestions.

\appendix

\section{Proof of Proposition \ref{prop:equivalence}}

\begin{proof}
For any $A \subset \mathcal{V}_x$, define,
\begin{equation}\label{eq:partition function2}
Z_{N, \beta}^{spin}(A)  = 
\Big  (  \,  \prod_{x \in \mathcal{V}} 
\int_{\mathbb{S}^{N-1}} d \varphi_x    \,   \Big  )  \,
\big ( \, \prod_{x \in A} \varphi^1_x \,  \big ) \, 
e^{- \beta H_{N}(\varphi)}.
\end{equation}
We will prove that, for any $N \in \mathbb{N}_{>0}$, $A \subset \mathcal{V}$, $\beta \geq 0$, 
under the choice of the weight function $U$ as in Proposition \ref{prop:equivalence}, we have that,
\begin{equation}
Z_{N, \beta}^{spin}(A) = Z_{N, \beta, U}(A).
\end{equation}
Thus,  by the definition of point-to-point function, Definition \ref{def:partition functions}, we will deduce Proposition \ref{prop:equivalence}.
The starting point of the expansion is the following identity, proved in \cite[Appendix A]{Chayes}, which holds for any $N \in \mathbb{N}_{>0}$ and  $n_1, n_2, \ldots n_N \in \mathbb{N}$,
\begin{equation}\label{eq:keypoint}
\int_{\mathbb{S}^{N-1}} 
(\varphi^{1})^{n_{1}}
\ldots
(\varphi^{N})^{n_{N}}
d \varphi
= 
\begin{cases}
 0  &\mbox{ \textit{if} }  n_{i} \in   2 \mathbb{N} + 1 \mbox{ \textit{for some} } i \in \{1, \ldots N \},\\
 \frac{\Gamma(\frac{N}{2}) \prod_{i=1}^{N} (n_{i} - 1)!! }{
 2^{\frac{n}{2}} \Gamma \big ((n+N)/2 \big )}
 & \mbox{ \textit{otherwise} },
\end{cases}
\end{equation}
where $d \varphi$ denotes the normalised uniform measure on $\mathbb{S}^{N-1}$, $n= n_1 + $ $\ldots$ $+n_N$, and $(n_i-1)!!$ is the double factorial, i.e. the number of ways to pair $n_i$ objects (hence $(-1)!!=1$).
Below, we will omit all sub-scripts to lighten the notation.
To begin, we re-write the exponential as follows,
\begin{equation}
\exp 
\Big \{   \beta
\sum\limits_{\{x,y\} \in \mathcal{E}} \varphi_x \cdot \varphi_y 
\Big \}
 \, = \, 
 \prod_{ \{x,y\} \in \mathcal{E}} \, \, 
 \prod_{i=1}^{N} \,  
 e^{  \beta \varphi_x^{i} \varphi_x^{i}}.
\end{equation}
For any $A \subset \mathcal{V}$, define
$$
\mathcal{M}_{\mathcal{G}} (A)  : = 
\{ m \in \mathcal{M}_{\mathcal{G}} \, \, : \, \, 
\forall x \in A, \sum_{e \in \mathcal{E} : x \in e}  m_e \, \,  \in 2 \mathbb{N}+1 \, \, \mbox{ and } \forall z \in \mathcal{V} \setminus A, \sum_{e \in \mathcal{E} : z \in e}  m_e \, \,  \in 2 \mathbb{N} \}.
$$
Now we expand as a Taylor series
and use  (\ref{eq:keypoint}) to restrict the sum to the terms which are not necessarily zero,
obtaining
\begin{multline}
Z^{spin}(A) = 
\sum\limits_{m^{1} \in \mathcal{M}_\mathcal{G}(A)}  \, 
\sum\limits_{m^{2} \in \mathcal{M}_\mathcal{G}(\emptyset)}
\ldots 
 \sum\limits_{m^{N} \in \mathcal{M}_\mathcal{G}(\emptyset)}
 \, 
 \Big ( 
 \prod_{e \in \mathcal{E}}    
 \frac{   \beta^{m_e^{1}+\dots+m_e^{N}}}{m_e^{1}! \ldots m_e^{N}!} 
 \Big ) \\
 \Big  (  \,  \prod_{x \in \mathcal{V}} 
\int_{\mathbb{S}^{N-1}} d \varphi_x    \,   \Big  )
\prod_{x \in \mathcal{V} \setminus A}
\Big ( \, 
(\varphi_x^{1})^{q_x^{1}}
\ldots 
(\varphi_x^{N})^{q_x^{N}}
\Big )
\prod_{x \in A } 
\Big ( \, 
(\varphi_x^{1})^{q_x^{1}+1}
(\varphi_x^{2})^{q_x^{2}}
\ldots 
(\varphi_x^{N})^{q_x^{N}}
\Big ).
\end{multline}
where we defined for any $x \in \mathcal{V}$, 
$
q^i_x(m) : = \sum_{e \in \mathcal{E} : x \in e} m^i_e.
$
We now rewrite the expression by first summing over all
$m \in \mathcal{M}_{\Gcal}(A)$ 
and $(m^1$, $m^2$, $\ldots$, $m^N)$,  such that, $m^1\in\Mcal_\Gcal(A)$, $m^i\in\Mcal_{\Gcal}(\emptyset)$, when $i\in\{2,\dots,N\}$,
 $m = \sum_{i=1}^{N} m^i$, 
and  $q_x = \sum_{i=1}^{N} q_x^i$, obtaining,
\begin{multline}
Z^{spin}(A) = 
\sum\limits_{m \in \mathcal{M}_{\mathcal{G}}(A)}  \,
\prod_{e \in \mathcal{E}} \Big ( \frac{\beta^{m_e}}{m_e!} \Big ) 
\sum\limits_{ \substack{ m^1 \in \mathcal{M}_\mathcal{G}(A) \\
 m^{i} \in \mathcal{M}_\mathcal{G}(\emptyset), i \geq 2 : 
 \\
 \sum_{i=1}^{N} m^i = m}} 
 \prod_{e \in \mathcal{E}} \Big (\frac{m_e!}{m_e^1! \ldots m_e^N!} \,   \Big )  \\ 
 \prod_{x \in \mathcal{V} \setminus A } 
 \Bigg ( \frac{\Gamma(\frac{N}{2})}{2^{\frac{q_x}{2}} \Gamma( (q_x+N)/2) }  \prod_{i=1}^{N} ( q^i_x - 1) !! \Bigg ) 
 \\  \prod_{x \in  A } 
   \Bigg (  \frac{\Gamma(\frac{N}{2})}{2^{\frac{q_x+1}{2}} \Gamma( (q_x+1 + N)/2) }   q_x^1!! \, \,  \prod_{i=2}^{N} ( q^i_x - 1) !!   \Bigg ) . 
\end{multline}
Above, the product right after the second sum can be interpreted as the number of colour assignments to the $m_e$ links which are parallel to the edge $e$ such that precisely $m_e^i$ links have colour $i$, for each $i=1, \ldots, N$.
Moreover, note that, if $q^i_x$ is an odd integer, then
$
q^i_x !!
$
is the number of ways $q^i_x$ links which are incident to $x$ can be ``paired" in such a way that only one link is unpaired and the remaining $(q^i_x-1)$ links are paired, while, if $q^i_x$ is an even integer, then
$(q^i_x-1)!!$ is the number of ways such $q^i_x$ links can be paired.
Thus, in the next step, 
 we replace the sum over $(m^i)_{i=1, \ldots, N}$ by the sum over $N$ possible colours for each link and the double factorial terms by the sum over all possible pairings of the links which are incident to each vertex. Recalling the definition of $n^i_x(m,c, \pi)$, which was given in (\ref{eq:numerofhits}), and putting 
 $n_x(m,c, \pi) = \sum_{i=1}^{N} n^i_x(m, c, \pi)$, we obtain that,
\begin{multline}
Z^{spin}(A) = \sum\limits_{ m \in \mathcal{M}_{\mathcal{G}}(A)} \, \prod_{e \in \mathcal{E}} \bigg ( \frac{\beta^{m_e}}{m_e!} \bigg )  \\  \sum_{c\in\Ccal_{\Gcal}(m)}\sum\limits_{\pi \in \mathcal{P}_{\mathcal{G}}(m,c) } \prod_{x \in \mathcal{V} \setminus A } \Bigg (   \frac{ \Gamma(N/2)}{2^{n_x(m,c,\pi)}\Gamma( n_x(m, c, \pi) + N/2)  }         \Bigg ) 
\prod_{x \in  A } \Bigg (   \frac{ \Gamma(N/2)}{ 2^{n_x(m,c, \pi)}\Gamma(n_x(m,c, \pi)  +  N/2)  }         \Bigg ) .
\end{multline}
In the previous expression we also used the fact that, 
if for a realisation $(m, c,  \pi) \in \mathcal{W}_{\mathcal{G}}(A)$,  $q_x$ links touch the vertex $x$, where  $x \in A$, this means that 
$q_x + 1 = 2 n_x(m,c, \pi)$.
Similarly, if 
for a realisation $(m, c, \pi) \in \mathcal{W}_{\mathcal{G}}(A)$,  $q_x$ links touch the vertex $x$, where $x \not\in A$, then $q_x = 2 n_x(m, c, \pi)$.
Plugging in the definition of the weight function $U_x$ (recall Definition \ref{def:measure} and the assumption of Proposition \ref{prop:equivalence}), the proof of Proposition \ref{prop:equivalence} is concluded.
\end{proof}

\end{document}